\documentclass[11pt,reqno,twoside]{article}

\usepackage{paper_style}

\usepackage{titlesec}
\titlespacing*{\section}{0pt}{1.125\baselineskip}{0.25\baselineskip}
\titlespacing*{\subsection}{0pt}{0.675\baselineskip}{0.125\baselineskip}
\titlespacing*{\subsubsection}{0pt}{0.5\baselineskip}{0.125\baselineskip}
\titlespacing*{\paragraph}{0pt}{0.25\baselineskip}{0.25\baselineskip}

\usepackage{lipsum} 

\newcommand{\email}[1]{\protect\href{mailto:#1}{#1}}

\begin{document}

\title{Adaptive Dimension Reduction for Overlapping Group Sparsity}
\author{Yifan Bai\thanks{School of Mathematical Sciences, Shanghai Jiao Tong University, Shanghai China
  (\email{yifanbai@sjtu.edu.cn}).}
\and Clarice Poon\thanks{Mathematics Institute, University of Warwick, Coventry UK 
  (\email{clarice.poon@warwick.ac.uk}, \url{https://cmhsp2.github.io/}).}
\and Jingwei Liang\thanks{School of Mathematical Sciences and Institute of Natural Sciences, Shanghai Jiao Tong University, Shanghai China 
  (\email{jingwei.liang@sjtu.edu.cn}, \url{https://jliang993.github.io/}).}}
\date{}
\maketitle

\begin{abstract}
Typical dimension reduction techniques for nonoverlapping sparse optimization involve screening or sieving strategies based on a dual certificate derived from the first-order optimality condition, approximating the gradients or exploiting certain inherent low-dimensional structure of the sparse solution.
In comparison, dimension reduction rules for overlapping group sparsity are generally less developed because the subgradient structure is more complex, making the link between sparsity pattern and the dual variable indirect due to the non-separability. 
In this work, we propose new dual certificates for overlapping group sparsity and a novel adaptive scheme for identifying the support of the overlapping group LASSO. We demonstrate how this scheme can be integrated into and significantly accelerate existing algorithms, including Primal-Dual splitting method, alternating direction method of multipliers and a recently developed variable projection scheme based on over-parameterization. 
We provide convergence analysis of the method and verify its practical effectiveness through experiments on standard datasets.
\end{abstract}




\section{Introduction}
Dimension reduction in large-scale sparse optimization has been an active research topic over the years, however most of the research focuses on the nonoverlapping cases, e.g. sparsity and nonoverlapping group sparsity.
While for overlapping group sparsity, it remains an open challenge. In this paper, we aim to address this problem. 
Specifically, we consider the following optimization problem,
\begin{equation}
\label{eq:main_problem}
\begin{aligned}
    &\min\limits_{x\in\mathbb{R}^n} \bBa{ \Phi(x) \eqdef \sfrac{1}{\lambda}F(x) + R(x) } \quad\text{with}\quad  F(x)=\sfrac{1}{2} \norm{Ax-y}^2 \qandq R(x) = \msum_{i=1}^{\Nn} w_i \norm{x_{G_i}} ,
\end{aligned}
\end{equation}
where $A: \RR^n \to \RR^m$ is a bounded linear operator, $y \in \RR^m$, 
$\Gg = \{ G_i\}_{i=1,...,\Nn}, \Nn\in\NN$ with $G_i\subset \dbrack{n}$ and $ \mcup_i G_i = \dbrack{n}$ represents the $\Nn$-group covering of $x$, and $w_i$ denotes the weight associated with $i$-th group. 
Problem \cref{eq:main_problem} emerges from numerous fields, such as data science, machine learning, signal/image processing and statistics, to name a few. 
Representative problems include LASSO \cite{tibshirani1996regression} (special case for $G_i = \{i\}$ and $w_i=1$), group LASSO \cite{yuan2006model} and overlapping group LASSO \cite{jacob2009group}, see also \cite{bach2012optimization} for more examples.

For large-scale optimization problem, a key bottleneck of algorithms to solve \eqref{eq:main_problem} efficiently is their ability to discard variables or groups that are inactive, i.e. the zero entries of optimal solution and the corresponding columns of $A$. 
For standard (nonoverlapping) sparsity, most dimension reduction techniques rely on the (group) separability which enables the construction of sharp dual certificates based on its subgradient, see \cref{sec: intro_dr} for details.
In comparison, overlapping group sparsity destroys this separability: the dual norm becomes a coupled decomposition over groups, the subgradient at zero is non-unique and shared across overlapping groups.

The aim of this work is to develop dimension reduction rules that explicitly account for this coupling. We design new dual certificates adapted to overlapping group structures and propose a novel adaptive dimension reduction scheme that can be integrated into standard optimization algorithms, yielding significant acceleration.
Before presenting our main contribution, we first briefly summarize numerical algorithms and existing dimension reduction techniques for solving \cref{eq:main_problem}.

\subsection{Numerical solvers}

First-order methods are among the most widely used approaches for solving sparse optimization problem, due to their simplicity and efficiency. 
An essential ingredient of first-order methods is the the proximal operator (see equation \cref{eq:prox}) of the nonsmooth function. 
When the proximal operator is easy to compute, popular choices of first-order methods to solve \cref{eq:main_problem} include Forward-Backward splitting method \cite{lions1979splitting} and its accelerated versions \cite{liang2017activity,beck2009fast}, the block coordinate descent \cite{friedman2010regularization}. 
When evaluating the proximal operator is difficult or expensive, alternatives include Primal-Dual splitting methods \cite{esser2010general,chambolle2011first,condat2012primal}, augmented-Lagrangian approaches \cite{gabay1976dual,glowinski1975approximation}. We refer to \cite{beck2017first,LaurentCondat2023} for overviews of first-order methods.

Despite their tremendous success over the past decades, first-order methods face growing scalability challenges: as problem sizes increase, their computational complexity increasingly undermines efficiency.
As solution of \cref{eq:main_problem} is sparse, incorporating such a sparse pattern into the first-order method would significantly reduce the computational complexity. 
Consequently, dimension reduction techniques, which detect the sparse pattern of the solution either as a preprocessing or over the course of iteration, are developed.

\subsection{Dimension reduction techniques}
\label{sec: intro_dr}

In general, existing dimension reduction approaches fall broadly into three categories: screening, sieving, and methods exploiting second-order sparsity. 
Below we summarize representative results on nonoverlapping sparsity.

\paragraph{Screening} 
For nonoverlapping sparsity, screening rules \cite{ghaoui2010safe} attempt to detect variables that are zero at the solution. When this identification is provably correct, the rule is safe. Notable examples include gap-safe screening \cite{fercoq2015mind,ndiaye2017gap,ndiaye2016gap}, see also the references therein for other approaches. Screening rules consist of two subcategories: the static rules which are performed once before deploying an optimization algorithms; the dynamic rules that operate sequentially along a regularization path (e.g. DPP \cite{wang2013lasso}) or dynamically during iterations \cite{fercoq2015mind,ndiaye2017gap}. 
We remark that, for dynamic safe screening rules the dimension is monotonically decreasing. 

\paragraph{Adaptive sieving}
Compared to screening, sieving methods take the opposite direction: i) at beginning, estimate the support of the solution via for example cross correlation test; ii) solve the problem restricted to the support; iii) apply optimality condition test and entries that violate the condition are added to the support estimation, then solve the restricted problem again. Note that the dimension of sieving approach is monotonically increasing, until the estimated support set includes the support of the solution. Representative work of adaptive sieving includes \cite{yuan2025adaptive,li2024efficient}.

\paragraph{Second-order sparsity} 
In recent years, semi-smooth Newton methods \cite{li2018highly,li2018efficiently,xiao2018regularized,zhang2020efficient} become increasingly popular. They reply on exploiting the semi-smoothness of for instance the proximal operator, and can achieve asymptotically superlinear convergence with low per-iteration complexity. 
However, for overlapping group structures, the proximal operator does not have a closed-form expression due to the non-separability, precluding the direct application of semi-smooth Newton methods.

\subsection{Contributions}
Compared to the active research of dimension reduction for nonoverlapping sparsity, attempts on overlapping sparsity are quite limited. 
To the best of our knowledge, existing dimension reduction methods for overlapping group sparsity include \cite{jenatton2011structured, yuan2011efficient, wang2013lasso,lee2014screening}, they either rely on special group structures or are tailored to the variants of \cref{eq:main_problem}.
To address this gap, we propose an {\tt Ada}ptive {\tt D}imension {\tt R}eduction for {\tt O}verlapping grou{\tt P} {\tt S}parsity (AdaDROPS, see \cref{algo:adadrops}), which is a generic framework for overlapping group sparse models. More specifically, our contribution consists of the following aspects
\begin{enumerate}
    \item We show that there are two choices of dual variables, which are the LASSO certificate (see \cref{def:lasso-cert}) and the OGN certificate (see \cref{def:ogn-cert}), which can certify the sparse pattern of overlapping group sparsity, see \cref{prop:lasso_certificate} and \cref{prop:ogn_certificate} respectively. This extends the current result on nonoverlapping sparsity to the overlapping setting.
    
    \item Based on the two dual certificates, we design an adaptive dimension reduction scheme for overlapping group sparse optimization problem, i.e. AdaDROPS and see \cref{algo:adadrops} for details. AdaDROPS is a versatile framework and can be applied to numerical methods for solving \cref{eq:main_problem}. Theoretical guarantee is also provided.
    
    \item We demonstrate the practical utility of AdaDROPS by integrating it with existing numerical schemes, including Primal-Dual splitting method \cite{chambolle2011first}, alternating direction method of multipliers \cite{gabay1983chapter,glowinski1975approximation} and variable projection \cite{poon2021smooth,poon2023smooth}. Extensive numerical experiments on standard datasets confirm that AdaDROPS provides significant computational acceleration, for certain cases more than an order acceleration is observed.
\end{enumerate}

\paragraph{Paper Organization} The rest of the paper is organized as follows: in \cref{sec:math_background}, we collect necessary mathematical definitions and notations. In \cref{sec:certificates}, we reformulate the overlapping group norm, and propose two dual certificates of overlapping group sparsity. In \cref{sec:adaptive_dr}, we propose the adaptive dimension reduction scheme for overlapping group sparsity (AdaDROPS). In \cref{sec:applications}, combination of AdaDROPS with existing numerical solvers is discussed. In \cref{sec:numerics}, we conduct numerical experiments to evaluate the performance of AdaDROPS. Finally, \cref{sec:conclusion} concludes the work with further discussions.

\section{Mathematical background}
\label{sec:math_background}

Given $n\in\mathbb{N}$, $\RR^n$ is the $n $-dimensional Euclidean space equipped with inner product $\langle\cdot,\cdot\rangle$ and induced norm $\norm{\cdot}$, $\bm{1}_n$ denotes $n$-dimensional vector of all $1$'s, $\Id_n$ denotes the identity operator on $\RR^n$.

Denote the set $\dbrack{n} = \{ 1,2,\cdots,n\}$, for an index set $G\subset \dbrack{n}$, $|G|$ denotes the cardinality of $G$. 
Given $x\in\RR^n$, $\mathrm{supp}(x) = \{ i \in \dbrack{n} \mid x_i \neq 0\}$ denotes the support of $x$, and $\diag(x)$ denotes the diagonal matrix whose diagonal elements are elements of $x$. 
We use $\odot$ to represent the Hadamard (point-wise) product of two vectors or matrices. For brevity, we write $x \odot x$ as $x^2$. We denote $x_1/x_2$ the point-wise division for two vectors $x_1$ and $x_2$ where $x_2$ has no zero elements.

Given a proper closed convex function $R: \RR^n \to ]-\infty, +\infty]$ 
and any $x\in\RR^n$, the subdifferential of $R$ at $x$ is a set defined by 
$$
\partial R :\RR^n \rightrightarrows \RR^n, x\mapsto \Ba{ u \in \RR^n \mid R(z) \geq R(x) + \langle  u,z-x\rangle } ,
$$
which is closed and convex. Any element of $\partial R(x)$ is called a subgradient. The {\it relative interior} of $\partial R$ is denoted as ${\rm ri}\pa{\partial R}$. 
The proximal operator of $R$ is defined as
\begin{equation}\label{eq:prox}
    \mathrm{prox}_{\gamma R}(z) = \mathop{\arg\min}_{x \in \RR^n} \left\{ \gamma R(x) + \sfrac{1}{2} \norm{x - z}^2\right\}, \enskip \gamma > 0  .
\end{equation} 
For $\ell_1$-norm and nonoverlapping $\ell_{1,2}$-norm, their proximal operator has closed-form expressions \cite{beck2017first}, which is not the case for overlapping $\ell_{1,2}$-norm due to the non-separability. 
{For indicator function of a closed convex set $\Tt\subseteq\RR^n$, its proximal operator is the projection operator of the set and denoted as $P_{\Tt}$. }
For the optimization problem \cref{eq:main_problem}, let $\xsol$ be a global minimizer, 
then its fixed-point characterization (according to the optimality condition) reads
\begin{equation}\label{eq:opt-cnd}
\xsol = \mathrm{prox}_{R} \bPa{ \xsol -  \sfrac{1}{\lambda}A^\top(A\xsol - y) } .
\end{equation}

\section{Certificates for overlapping group LASSO}\label{sec:certificates}
In this section, we first introduce a lifted
representation of the overlapping group norm and the associated extended supports and subspaces; 
Then define two certificates: a gradient-based certificate $\betasol$ (easy to compute, conservative), and a lifted sub-gradient certificate $u^\dagger$ (tailored to overlapping setting). We keep the proofs that explain the tightness gaps in the main text, while collecting properties of the lifted operators in \cref{sec:ogn_properties}.

\subsection{Overlapping group norm}\label{sec:overlapping_group_norm}
In this part, we provide the lifted reformulation of overlapping group norm, the geometric objects (extended supports and subspaces) 
that will be used to construct \emph{dual certificates}. 

\subsubsection{Lifted formulation}\label{sec:def_L}

To analyze the overlapping group sparsity, it is convenient to rewrite the regularizer through a \emph{lifting operator} that makes the group structure explicit. The key is to embed $x\in \RR^n$ into a higher-dimensional vector $z = Lx \in \RR^p$ whose coordinates are arranged in \emph{nonoverlapping groups} corresponding to the original overlapping groups. In the lifted space, the overlapping group norm becomes a mixed $\ell_{1,2}$-norm, which yields a cleaner description of its subdifferential and clarifies how to build \emph{dual certificates}.

\paragraph{Overlapping group structure}
Given $n,\Nn\in\NN$, an (overlapping) $\Nn$-group covering of $\dbrack{n}$ is the set $\Gg = \big\{G_1,...,G_{\Nn} \big\}$ such that
\begin{equation}\label{eq:partition}
G_i \subseteq \dbrack{n}, ~ i\in\dbrack{\Nn}  
,\quad
\mcup_{i\in\dbrack{\Nn}} G_i = \dbrack{n} ,
\end{equation}
and the overlap of any two groups is not necessarily empty. 

\paragraph{Lifting operator}
Given $x\in\RR^n$, for each group $G_i \in \Gg$, define the linear operator $L_{G_i}: \RR^n \to \RR^{|G_i|}$ by
$
 L_{G_i} x =w_i x_{G_i}. 
$
Stacking all $L_{G_i}, i\in\dbrack{\Nn}$ yields 
\begin{equation}
    L:\RR^n \to \RR^p, \quad p = \sum_{i\in\dbrack{\Nn}} |G_i|,\quad Lx = \Pa{L_{G_i}x}_{i=1}^{\Nn}.
\end{equation}
In overlapping setting, the matrix form $L$ has full column rank.

\paragraph{Nonoverlapping group structure in the lifted space}
Given $z\in\RR^p$, let $\Jj \eqdef \{J_1,\cdots, J_{\Nn}\}$ be a \emph{nonoverlapping} $\Nn$-group {partition} of $\dbrack{p}$ such that
\begin{equation}\label{eq:zJi_eq}
    \forall i \in \dbrack{\Nn}, \quad z_{J_i} = L_{G_i} x . 
\end{equation}
Given any $k \in \dbrack{p}$, define 
\begin{equation}\label{eq:phi_k}
\phi(k) = t \in\dbrack{\Nn} \enskip{\rm such~that}\enskip k \in J_t ,
\end{equation}
that retrieves the nonoverlapping group $J_t$ that $k$ belongs to; see \cref{sec:fact:L} for more properties of $L$ and $\phi$.

\paragraph{Reformulation of overlapping group norm}
The overlapping group norm of $x$ now can be represented as
\[
R(x) 
= \msum_{i\in\dbrack{\Nn}} w_i \norm{x_{G_i}}
= \msum_{i\in\dbrack{\Nn}} \norm{z_{J_i}}
= \norm{Lx}_{1,2} .
\]

\paragraph{Subdifferential of overlapping group norm}
The subdifferential of $R$ is $\partial R(x) = L^\top \partial \norm{Lx}_{1,2}$, where for any $u \in\norm{Lx}_{1,2} $, we have
\begin{equation}
\label{eq:subdiff_l2}
u_{J_i}
= \left\{
\begin{aligned}
    & x_{G_i} /\norm{x_{G_i}}, &\quad x_{G_i} \neq 0 , \\
    &\ba{ g\in\RR^{\abs{{G_i}}} \mid \norm{g}\leq 1 } , &\quad x_{G_i} = 0 .
\end{aligned}
\right.
\end{equation}
Note that \cref{eq:subdiff_l2} shows that the subgradient can \emph{certify group support}: active groups must satisfy $\norm{u_{J_i}}=1$, while $\norm{u_{J_i}} < 1$ is a sufficient condition for $x_{G_i}=0$.

\subsubsection{Active groups, supports and subspaces}\label{sec:def_extended_supp}
Let $x\in\RR^n$ be overlapping group sparse, then $z=Lx$ is group sparse. Define the following index sets
\begin{itemize}
\item Index of the active (nonzero) groups
\begin{equation}\label{eq:support_Ix}
\Ii_{x} \eqdef \Ba{ t\in\dbrack{\Nn} \mid \norm{x_{G_t}} \neq 0 }. 
\end{equation}

\item {Extended {\it coordinate} support of $x$ and the corresponding spanned subspace}
\begin{equation}\label{eq:support_Ex}
\Ee_{x} \eqdef \dbrack{n} \setminus \Pa{ \mcup_{t \in \Ii_{x}^c} G_t } ~~~{\rm and}~~~  \Tt_{x} \eqdef \Ba{ x' \in \RR^n \mid \supp(x') \subseteq \Ee_{x} } .
\end{equation}
By ``extended" it means that $\Ee_x$ contains the support of $x$ (i.e. $\supp(x) \subseteq \Ee_{x} \subseteq \mcup_{t \in \Ii_{x}} G_t$), and remains compatible with those active groups (coordinates touched by any inactive groups are excluded). 

\item {Extended {\it group} support of $z$ and spanned subspace}
\begin{equation}\label{eq:support_Ez}
\Ee_{z} \eqdef \mcup_{t \in \Ii_{x}} J_t   ~~~{\rm and}~~~  \Tt_{z} \eqdef \Ba{ z' \in \RR^p \mid \supp(z') \subseteq \Ee_{z} }  .
\end{equation}
Clearly, there holds $\supp(z) \subseteq \Ee_{z}$. 

\item {Extended lifted coordinate support and spanned subspace}
\begin{equation}\label{eq:support_EL}
\Ee_{L} \eqdef \supp\pa{LP_{\Tt_x} \bm{1}_n} \subseteq \dbrack{p}  ~~~{\rm and}~~~      \Tt_{L} \eqdef \Ba{z' \in \RR^p \mid \supp(z') \subseteq \Ee_L} ,
\end{equation}
which describe active coordinates in the lifted space  corresponding to the extended support after lifting.

\end{itemize}
An illustration of these notions is provided in \cref{fig:index_sets}. 
Properties of the projection operators of these subspaces are collected in \cref{sec:prop:proj}.

\begin{figure}[H]
    \centering
    \includegraphics[height=47mm, trim={4mm 1mm 4mm 3mm},clip]{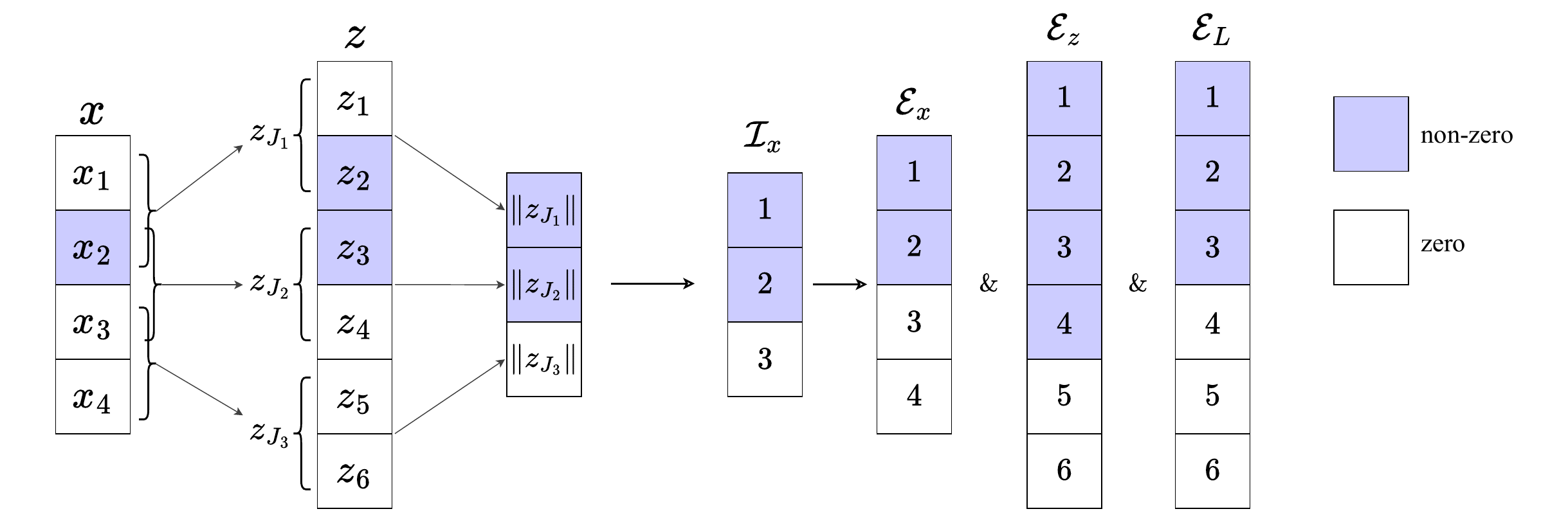} 
    \vspace*{-1ex}
    \caption{
    Illustration of overlapping group norm: i) $x\in\RR^4$ has $1$ nonzero element, the grouping is $\Gg = \Ba{ \ba{1,2}, \ba{2,3}, \ba{3,4} }$; ii) $z\in\RR^6$ has nonoverlapping grouping $\Jj = \Ba{ \ba{1,2}, \ba{3,4}, \ba{5,6} }$; iii) The active groups of $x$ is $\Ii_x = \Ba{1,2}$; iv) The extended supports are $\Ee_x = \Ba{1,2},\Ee_z = \Ba{1,2,3,4}$ and $ \Ee_L = \Ba{1,2,3}$. 
    } \vspace*{-2ex}
    \label{fig:index_sets}
\end{figure}

\begin{remark}
The reason of considering the extended support instead of the support is that we focus on dealing with the group sparsity. However, it reduces to the standard support when considering $\ell_1$-norm. 
\end{remark}

\subsubsection{Effective lifting operator}
A recurring issue in the overlapping setting is that directions orthogonal to the 
space $\Tt_x$ can still create components in the active lifted groups after lifting. The following construction removes precisely this ``leakage" while preserving first order optimality condition.

\begin{definition}[Effective lifting operator]
For a given $x\in \RR^n$ and its associated subspaces $\Tt_x, \Tt_z$, the effective lifting operator is defined by
\begin{equation*}
    \widehat{L} \triangleq L - P_{\Tt_{z}} L P_{\Tt_{x}^\bot}.
\end{equation*}
\end{definition}

An illustration of $\widehat{L}$ is provided below in \cref{fig:hatL}, note that $\widehat{L}$ preserves the full column rank; see \cref{rem:widehat_eq}. 

\begin{figure}[htb]
    \centering
    \includegraphics[height=51mm, trim={8mm 3mm 8mm 6mm},clip]{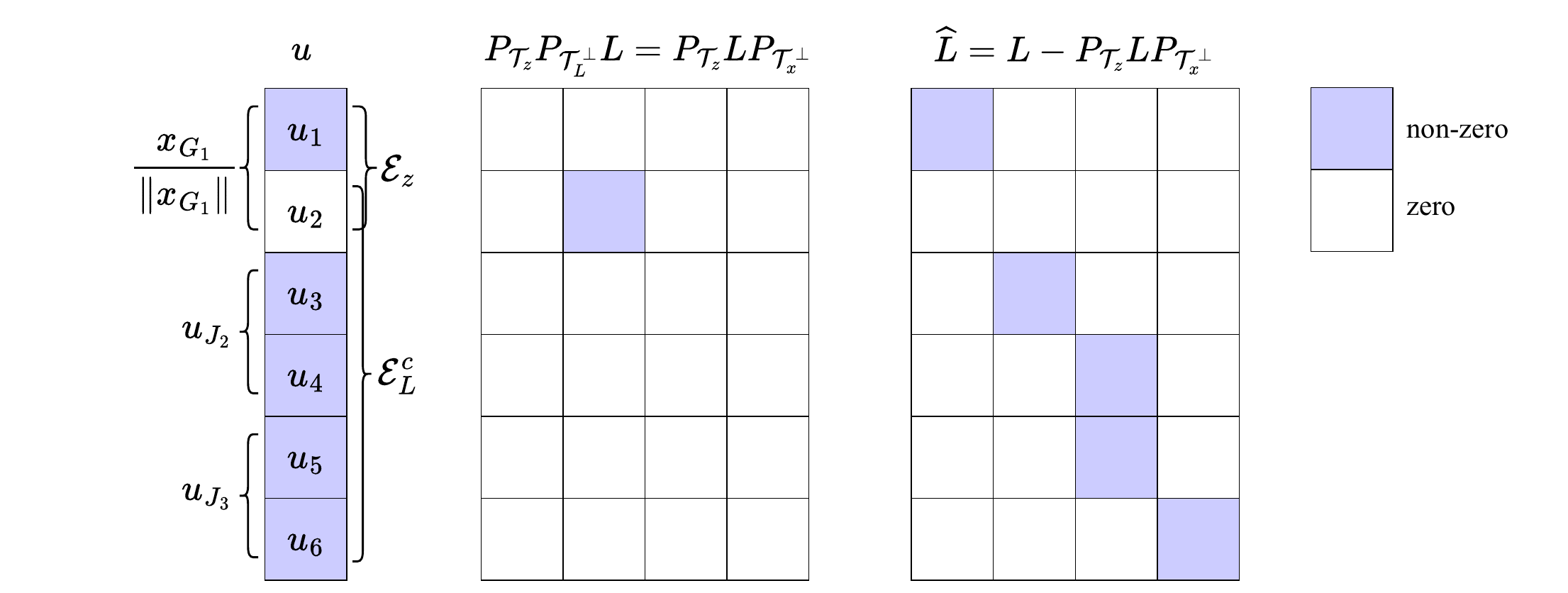}
    \caption{Illustration of $\widehat{L}$. Consider $\Gg = \Ba{ \ba{1,2}, \ba{2,3}, \ba{3,4}}$ with $x_{G_1} \neq 0, x_{G_2}=x_{G_3}=0$. For any $u\in \partial \norm{Lx}_{1,2}$, we have $u_2=0$. The extended support $\Ee_z = \ba{1,2}, \Ee_L = \ba{1}$. The subtraction in $\widehat{L}$ is removing the 2nd row (corresponding to $x_2$, which overlaps between zero and nonzero groups).}
    \label{fig:hatL}
    \vspace*{-1ex}
\end{figure}

In \cref{sec:ogn_certificate} we use this operator to define a dual certificate for support identification. 
In what follows, we show that the components removed in $\widehat{L}$ play no role when pairing $L$ with subgradient.

\begin{proposition}\label{rem:widehat_eq}
There holds
\begin{itemize}
    \item $\widehat{L}$ has full column rank.

    \item For any $u\in\partial \norm{Lx}_{1,2}$, one has $(P_{\Tt_z} L P_{\Tt_x^\bot})^\top u=0$ and
    \begin{equation*}
        L^\top u = \widehat{L}^\top u.
    \end{equation*}
\end{itemize}
    
\end{proposition}

We refer to \cref{appendix:widehat} for the proof.

\subsection{Dual certificates}
\label{sec:certificate}

In sparse optimization, the key of dimension reduction is an easy access dual vector which can certify the sparse pattern of the solution. Such a vector is called the {\it dual certificate} \cite{Duval2016SparseSS,Vaiter13}. 
For nonoverlapping sparsity regularizations, such as $\ell_1$-norm and nonoverlapping group norm, constructing the dual certificate is rather simple and has been extensively explored in dimension reduction and sensitivity analysis; see for instance \cite{ndiaye2017gap,Duval2016SparseSS} and the references therein. 

However, due to the non-diagonal and full-column-rank $L$, there is limited work on extending existing results to the overlapping group norm. 
In the overlapping group LASSO,
there are two natural certificate viewpoints:
\begin{enumerate}
    \item A \emph{primal-space} certificate $\betasol$ obtained from the gradeint of the date fidelity term, which is often used in safe screening;
    \item A \emph{lifted-space} certificate $u$ that directly lives in $\partial \norm{L\xsol}_{1,2}$, hence certifies group activity via group norms.
\end{enumerate}
In the following, we discuss two certificates based on these two viewpoints.

\subsubsection{Two certificates}

Two choices of dual certificates are adopted in this work: the LASSO certificate (see \cref{def:lasso-cert}) and the OGN certificate (see \cref{def:ogn-cert}). 
The former one inherits the existing work \cite{Vaiter13,ndiaye2017gap,yuan2025adaptive} on nonoverlapping sparsity, while the latter one is designed based on the structure of the subdifferential.

\begin{definition}[LASSO certificate]
\label{def:lasso-cert}
    Define the LASSO certificate of \cref{eq:main_problem} as
    \begin{equation}
        \betasol \eqdef   - \sfrac{1}{\lambda} \nabla F(\xsol) \quad \Pa{\text{for least squares: } \betasol = - \sfrac{1}{\lambda}A^\top ( A \xsol - y)}.
    \end{equation}
\end{definition}

For \cref{eq:main_problem}, based on the optimality condition and \cref{rem:widehat_eq}, there exists $\usol \in \partial\norm{L\xsol}_{1,2}$ such that
\begin{equation*}
    L^\top \usol = \betasol 
    \qandq
    \widehat{L}^\top \usol = \betasol  .
\end{equation*}
The identity above motivates constructing a canonical lifted certificate (see \cref{sec:ogn_certificate}) by solving the linear system with $\widehat{L}$.
%
%
However, as $\widehat{L}^\top$ is only full-row-rank,  
we can consider the minimal norm solution of $\widehat{L}^\top u = \betasol$ which is
\begin{equation}\label{eq:umin}
\widehat{u}_{\min} = \widehat{L} ( \widehat{L}^\top \widehat{L})^{-1} \betasol  . 
\end{equation}
We then define the following {\it OGN certificate} (with OGN stands for overlapping group norm) based on this minimal norm solution.

\begin{definition}[OGN certificate]
\label{def:ogn-cert}
The OGN certificate $u^\dagger$ is defined by
\[
\forall t\in \dbrack{\Nn},\quad 
u^\dagger_{J_t}
= 
\left\{
\begin{aligned}
&(\widehat{u}_{\min})_{J_t} , & t \in \Ii_{\xsol}^c , \\
&{\xsol_{G_t}}/{ \norm{\xsol_{G_t}} }, & t \in \Ii_{\xsol} .
\end{aligned}
\right.
\] 
\end{definition}

More explanations of the construction and certification property of the OGN certificate are elaborated in \cref{sec:ogn_certificate}.

\begin{remark}[Nonoverlapping case]
If the groups do not overlap, $L$ is diagonal and the two certificates coincide, recovering the standard sparsity and group sparsity screening rules.
\end{remark}

\subsubsection{The LASSO certificate: a conservative rule}\label{sec:lasso_certificate}

In the nonoverlapping setting, $\betasol$ can be directly used to certify the sparsity of $\xsol$ in the sense that
\begin{equation}\label{eq:lasso_certificate}
\forall t\in\dbrack{\Nn} ~~
\left\{
\begin{aligned}
\norm{\betasol_{G_t}} < w_t  &  \enskip\Longrightarrow\enskip   \xsol_{G_t} = 0  , \\
\xsol_{G_t} \neq 0  & \enskip\Longrightarrow\enskip \norm{\betasol_{G_t}} = w_t   .
\end{aligned}
\right.   
\end{equation}
The rule \cref{eq:lasso_certificate} is the foundation of existing work on dimension reduction for nonoverlapping sparsity. 
Due to this reason, we call $\betasol$ the LASSO certificate.

It turns out that \cref{eq:lasso_certificate} can be extended to overlapping group sparsity. 

\begin{proposition}
\label{prop:lasso_certificate}
    For problem \cref{eq:main_problem} with overlapping group norm, for any $t \in \dbrack{ \Nn }$, if $\norm{\betasol_{G_t}} < w_t$, then $\xsol_{G_{t}} = 0$. 
\end{proposition}

\begin{proof}
Let $\usol$ be such that $\usol \in \partial \norm{L\xsol}_{1,2}, \betasol = L^\top\usol$. Before proving the result, we need a new definition: for any $i\in\dbrack{n}$, define 
\begin{equation}\label{eq:Ki}
\mathscr{K}_i \eqdef \Ba{ k\in\dbrack{p} \mid L_{k,i} \neq 0 } ,
\end{equation}
for which we have 
\begin{enumerate}[label={\rm \alph{*}).}]
    \item The cardinality $\abs{\mathscr{K}_i}$ represents the number of groups that $\xsol_i$ is assigned to.
    \item $\usol_k(k\in\mathscr{K}_i)$ are the entries of $\usol$ that contribute to $\betasol_i$. 
\end{enumerate}
Moreover, recall the the group number retrieval mapping $\phi$ defined in \cref{eq:phi_k}, which will be used later.

Given $t\in\Ii_{\xsol}$ and $i\in {G_t}$ with $\xsol_i\neq 0$, let $k_i\in\dbrack{p}$ be such that $k_i \in \mathscr{K}_i$ and $\phi(k_i)=t$, we have the following decomposition of the $\betasol_i$
\begin{equation*}
\betasol_i
= \msum_{ k\in \mathscr{K}_i, j = \phi(k) } w_j \usol_k 
= w_t \usol_{k_i}  + \msum_{ k\in \mathscr{K}_i\setminus \ba{k_i}, j = \phi(k)  } w_j \usol_k .
\end{equation*}
Therefore, for the group $\betasol_{G_t}$, we get
\begin{equation}\label{eq:beta_Gt}
\begin{aligned}
\norm{\betasol_{G_t}}^2 
&= 
\msum_{i\in G_t, \xsol_i \neq 0} \pa{ \betasol_{i} }^2 + \msum_{i\in G_t, \xsol_i = 0} \pa{ \betasol_{i} }^2 \\
&= \msum_{i\in G_t, \xsol_i \neq 0} \Pa{ w_t \usol_{k_i}  + \underbrace{\msum_{ k\in \mathscr{K}_i\setminus \ba{k_i}, j = \phi(k)  } w_j \usol_k}_{T_1}  }^2 + \underbrace{\msum_{i\in G_t, \xsol_i = 0} \pa{ \betasol_{i} }^2}_{T_2} \\
&\geq
\msum_{i\in G_t, \xsol_i \neq 0} \Pa{ w_t \usol_{k_i}  + T_1  }^2 
\overset{\text{\ding{172}}}{\geq} \msum_{i\in G_t: \xsol_i \neq 0} \pa{ w_t \usol_{k_i}   }^2 \\
&= w_t^2 \msum_{i\in G_t,\xsol_i\neq 0} \pa{  \usol_{k_i} }^2 
= w_t^2 \left\lVert\frac{\xsol_{G_t}}{\norm{\xsol_{G_t}}}\right\rVert^2 = w_t^2 .
\end{aligned} 
\end{equation}
The inequality ${\text{\ding{172}}}$ comes from 
the fact that given an $i\in\dbrack{n}$ with $\xsol_{i}\neq 0$, $\usol_{k_i}$ and $\usol_{k}, k\in \mathscr{K}_i$ have the same signs. 
Consequently, we reach the following conclusion
\[
    \xsol_{G_{t}} \neq 0
    \enskip\Longrightarrow\enskip
    \norm{\betasol_{G_t}}  \geq w_t .
\]
The proved contrapositive leads to the desired claim.
\end{proof}

In the above proof, the terms $T_1$ and $T_2$ are discarded directly, we refer to \cref{appendix:relaxation} for further discussion on their influence to the inequality \cref{eq:beta_Gt}. 

\begin{remark}[Nondegeneracy and tightness]\label{remark:degeneracy}
In the {\it nonoverlapping} setting, if $\betasol \in {\rm ri}(\partial R(\xsol))$, i.e. $\betasol$ is nondegenerate, \cref{eq:lasso_certificate} becomes
\[
\forall t\in\dbrack{\Nn} ~~
\left\{
\begin{aligned}
\norm{\betasol_{G_t}} < w_t  &  \enskip\Longleftrightarrow\enskip   \xsol_{G_t} = 0  , \\
\xsol_{G_t} \neq 0  & \enskip\Longleftrightarrow\enskip \norm{\betasol_{G_t}} = w_t   .
\end{aligned}
\right.   
\]
However, for the {\it overlapping} case, characterizing the nondegeneracy of $\betasol$ is complicated and depends on the overlapping patterns. 
For example, one may have $\xsol_{G_t} = 0$ for $\norm{\betasol_{G_t}} > w_t$. 
This difference implies that for overlapping sparsity, \cref{prop:lasso_certificate} is not as tight as that for the nonoverlapping case.  This is why $\betasol$ is viewed as a conservative screening tool in the overlapping setting.

\end{remark}

\begin{remark}
    A similar result to \cref{prop:lasso_certificate} can be found in \cite{lee2014screening}. 
    The main focus of \cite{lee2014screening} is to design static and sequential safe screening rule for overlapping group sparsity, which is a pre-processing of data to remove useless entries. While our goal is to design adaptive dimension reduction, which iteratively estimate the support of optimal solution. See \cref{sec:related_work} for more explanations of \cite{lee2014screening}. 
\end{remark}

\subsubsection{The OGN certificate: direct certification in the lifted space}\label{sec:ogn_certificate}

The definition of LASSO certificate suggests there exists $\usol \in \partial \norm{L\xsol}_{1,2}$
such that $\betasol = L^\top \usol$. 
According to \cref{eq:subdiff_l2}, it is clear that $\usol$ can certify the sparsity of $\xsol$ in the sense 
\begin{equation}\label{eq:pre_certificate}
\forall i\in\dbrack{\Nn} ~~
\left\{
\begin{aligned}
\norm{\usol_{J_t}} < 1 & \enskip\Longrightarrow\enskip   \xsol_{G_t} = 0  , \\
\xsol_{G_t} \neq 0  & \enskip\Longrightarrow\enskip  \norm{\usol_{J_t}} = 1   ,
\end{aligned}
\right.
\end{equation}
implying $\usol$ can also serve as a certificate. 
However, finding such a $\usol$ is not obvious since $L^\top$ is under-determined, therefore we need a proper approach to approximate $\usol$, leading to the definition of OGN certificate $u^\dagger$ in \cref{def:ogn-cert}. To show that $u^\dagger$ indeed can be used as a certificate, we have the following results. The proofs in this section can be found in \cref{appendix:ogn_cert}.

\begin{proposition}[OGN certificate]
\label{prop:ogn_certificate}
Given the LASSO certificate $\betasol$ and OGN certificate $u^\dagger$, there holds $L^\top u^\dagger = \betasol$.
Moreover, for any $t\in\dbrack{\Nn}$, $\xsol_{G_t} = 0$ if $\norm{u^\dagger_{J_t}} < 1$.
\end{proposition}

Moreover, $u^\dagger$ can reduce the influence caused by the overlapping entries compared to the LASSO certificate, revealed in the following inequality.

\begin{proposition}[Link between the two certificates]
\label{prop:link_two}
    For any $t \in \dbrack{\Nn}$, there holds $\norm{u^\dagger_{J_t}} \leq \norm{\betasol_{G_t}}/w_t$.
\end{proposition}

If $t \in \Ii_{\xsol}$, according to the \cref{eq:beta_Gt} and the definition of $u^\dagger$, we have
\begin{equation*}
    u^\dagger_{J_t} 
    = \Pa{\usol_{k_i}}_{i\in G_t, \xsol_i \neq 0} 
    = {\xsol_{G_t}} /{\norm{\xsol_{G_t}}} .
\end{equation*}
Hence the value of $u^\dagger$ on the groups in $\Ii_{\xsol}$ is exactly the term after discarding $T_1,T_2$, leading to $\norm{u^\dagger_{J_t}} \leq \norm{\betasol_{G_t}}/w_t$. Proof for the case $t \in \Ii_{\xsol}^c$ can be found in \cref{appendix:ogn_cert}.

\begin{remark}[Tightness of \cref{prop:ogn_certificate}]
Theoretically, $\norm{\betasol_{G_t}} < w_t$ means $\norm{u^\dagger_{J_t}} < 1$ for $t \in \Ii_{\xsol}^c$. 
As we shall see in the example \cref{example:inequality} and the comparison of \cref{cmp:comparison}, $u^\dagger$ is much more effective than $\betasol$, this is also validated by the numerical experiments in \cref{sec:numerics:overlapping}. 
\end{remark}

\begin{remark}[Why $u^\dagger$]
\begin{itemize}
    \item As $\widehat{u}_{\min}$ is only a minimal norm solution, $\widehat{u}_{\min} \in \partial \norm{L\xsol}_{1,2}$ does not hold in general, meaning that there may exist $t\in\Ii_{\xsol}$ such that $\norm{(\widehat{u}_{\min})_{J_t}} \neq 1$, or $t\in\Ii_{\xsol}^c$ such that $\norm{(\widehat{u}_{\min})_{J_t}} > 1$. Fixing $\norm{(\widehat{u}_{\min})_{J_t}} > 1, t\in\Ii_{\xsol}^c$ is difficult, while correcting $\norm{(\widehat{u}_{\min})_{J_t}} \neq 1, t\in\Ii_{\xsol}$ is straightforward which is the reason of defining $u^\dagger_{J_t} = \xsol_{G_t}/\norm{\xsol_{G_t}}$ for $t \in \Ii_{\xsol}$. 
    
    \item The minimal norm solution ensures the inequality \cref{prop:link_two} for those groups in $\Ii_{\xsol}^c$. We present an example of the inequality in \cref{example:inequality}.

    \item Using the effective lifting operator $\widehat{L}$ to compute the minimal norm solution $\widehat{u}_{\min}$ rather than $L$ is for the purpose to ensure the optimality condition $L^\top u^\dagger = \betasol$ in \cref{prop:ogn_certificate}.
\end{itemize}
\end{remark}

\begin{example}\label{example:inequality}
    We give an example to illustrate \cref{prop:link_two}. Suppose the overlapping group of $x$ is $\Gg = \Ba{ \ba{1,2,3}, \ba{2,3,4}, \ba{3,4,5} }$, and the nonzero group is $\Ii_{\xsol} = \Ba{1}$. 
    The LASSO certificate and the OGN certificate on the 2nd group are
\begin{equation*}
\betasol_{G_2} = 
\begin{pmatrix}
        \frac{1}{w_2} \betasol_2 &
        \frac{1}{w_2} \betasol_3 &
        \frac{1}{w_2} \betasol_4 
\end{pmatrix}^\top 
\qandq
    u^\dagger_{J_2} =
    \begin{pmatrix}
        \frac{1}{w_2} \betasol_2 &
        \frac{w_2}{w_2^2 + w_3^2} \betasol_3 &
        \frac{w_2}{w_2^2 + w_3^2} \betasol_4 
    \end{pmatrix}^\top . 
\end{equation*}
When $w_2 = w_3$, the second and third entry of $\betasol_{G_2}$ are twice the corresponding entries of $u^\dagger$, since indices $3$ and $4$ each appear twice among the zero groups. The more frequently an element overlaps with other zero groups,
the more $u^\dagger$ outperforms $\betasol$.
\end{example}

\begin{remark}
Note that our discussion on the two certificates does not rely on the structure of $F(x)$ in \cref{eq:main_problem}, hence the above result can be extended to other cases including logistic regression or squared hinge loss. 
\end{remark}

\subsubsection{Numerical example}\label{cmp:comparison}

To compare the effectiveness of the two certificates, 
we consider a numerical comparison of the settings
\begin{itemize}
\item[i).] Fix group size as $10$, and number of groups $\Nn=100$. The overlapping size is chosen from $\ba{1,2,...,6}$.

\item[ii).] Choose $m={\tt round}(n/2)$, let $A\in\RR^{m\times n}$ be sampled from normal distribution and take $y \in \RR^m$. 

\item[iii).] Set group weight as $w_i = \sqrt{10}$, and $\lambda>0$ is properly chosen such that the solution of \cref{eq:main_problem} has around $10$-$15$ nonzero groups. 
\end{itemize} 
In \cref{tab:overlapp_size} we provide the number of zero groups  detected by \cref{prop:lasso_certificate} and \cref{prop:ogn_certificate}. 
It can be observed that with overlapping size increasing, the number of zero groups detected by \cref{prop:lasso_certificate} decreases, especially when overlapping size is larger than half group size. 
In comparison, $u^\dagger$ yields a robust performance, implying its potential in more realistic scenarios. 
We remark that the two certificates are based on the optimal solution $\xsol$. When they are estimated over the course of an iterative scheme, comparison can be found in \cref{sec:numerics:overlapping}.

\begin{table}[H]
\centering
\caption{Number of zero groups detected by Propositions \ref{prop:lasso_certificate} \& \ref{prop:ogn_certificate} under different overlapping size, with fixed group size $10$, number of groups $\Nn=100$, and overlapping size changing from $1$ to $6$.}\vspace*{-1ex}
\tabulinesep=1mm
\begin{tabu}{|c|c|c|c|c|c|c|}
\hline
overlapping size & 1 & 2 & 3 & 4 & 5  & 6 \\ \hline
number of zero groups & 88 & 86 & 89 & 89 & 85 & 85 \\ \hline
detected by \cref{prop:lasso_certificate} & 79 & 60 & 53 & 43 & 4 & 1 \\ \hline
detected by \cref{prop:ogn_certificate} & 87 & 81 & 87 & 87 & 73 & 83  \\ \hline
\end{tabu}
\vspace*{-1ex}
\label{tab:overlapp_size}
\end{table}


\section{Adaptive dimension reduction for overlapping group sparsity}
\label{sec:adaptive_dr}

\cref{prop:lasso_certificate} and \cref{prop:ogn_certificate} imply that we can use the two certificates to design dimension reduction scheme for (overlapping group) sparse optimization. 
Motivated by the idea of {\it adaptive sieving} developed in~\cite{yuan2025adaptive}, we propose an {\tt Ada}ptive {\tt D}imension {\tt R}eduction for {\tt O}verlapping grou{\tt P} {\tt S}parsity (AdaDROPS); see \cref{algo:adadrops} for details.

\subsection{Derivation of AdaDROPS}
As dimension reduction is an algorithm-agnostic add-on, we use {\tt algorithm} to denote a numerical solver and refer to \cref{sec:applications} for examples. 
Let $\{\xk\}_{k\in\NN}$ be the minimizing sequence generated by {\tt algorithm}, the basic idea of AdaDROPS is: at step $k$,  apply \cref{prop:lasso_certificate} or \cref{prop:ogn_certificate} to $\xk$ to estimate the sparsity of $\xsol$, then solve the optimization problem restrict to the estimated support. 
In practice, this achieves significant reduction of both dimension and computational complexity.

Before presenting AdaDROPS, we need the following notations
\begin{itemize}
\item $\xsol$ is a solution of \cref{eq:main_problem}, according to \cref{eq:support_Ix}, \cref{eq:support_Ex} and \cref{eq:support_Ez} denote $\Ii_{\xsol}$ the support of nonzero groups, and $\Ee_{\xsol}, \Ee_{\zsol}$ the extended supports of $\xsol, \zsol$. 

\item 
Given $\xk$ and an estimated $\Ii_x^{(k)}$, denote $\Ee_x^{(k)}, \Ee_z^{(k)}$ according to \cref{eq:support_Ex} and \cref{eq:support_Ez}. The associated subspaces are denoted as $\Tt_x^{(k)}$ and $\Tt_z^{(k)}$.

\item 
Define the restricted operators
$
A_{k} = A \circ P_{{\Tt}_{x}^{(k)}}
~{\rm and}~
L_{k} = L \circ P_{{\Tt}_{x}^{(k)}} . 
$
For the original problem \cref{eq:main_problem}, its restriction to ${\Tt}_{x}^{(k)}$ is reads
\begin{equation}
\label{eq:restricted_problem_k}
    \min_{x \in \RR^{n}}  \bBa{ \Phi_{k}(x) \eqdef
    \sfrac{1}{2\lambda} \norm{A_k x - y}^2 +  \norm{L_{k}x}_{1,2} } .
\end{equation}

\end{itemize}

\begin{remark}
\label{rem:proj_diag}
    Note that in \cref{eq:restricted_problem_k}, $x\in\RR^n$ and is not restricted to $\RR^{\abs{\Ee_{x}^{(k)}}}$, the main reason of doing so is to avoid rearrangement of the grouping. Moreover, it still provides computational reduction, as $A_{k}, L_{k}$ have $n-\abs{\Ee_{x}^{(k)}}$ zero columns due to $P_{{\Tt}_{x}^{(k)}}$.
\end{remark}

AdaDROPS is a composition of the {\tt algorithm} and sparsity estimation: i) given a current estimate $\Ii_{x}^{(k)}$, solve the restricted problem \cref{eq:restricted_problem_k} via {\tt algorithm} which outputs $\xkp$; ii) Based on $\xkp$, compute a new estimate $\Ii_x^{(k+1)}$ by $\betakp$ (or $\ukp$) according to \cref{prop:lasso_certificate} (or \cref{prop:ogn_certificate}), then return to step ``i)''. 
Elaborating these two steps leads to the following framework.

\begin{algorithm}[H]
    \renewcommand{\algorithmicrequire}{\textbf{Input:}}
	\renewcommand{\algorithmicensure}{\textbf{Output:}}
    \small
    \caption{The AdaDROPS framework}
    \begin{algorithmic}[1]
    \label{algo:adadrops}
        \REQUIRE  {initial estimated support $\Ii_{x}^{(0)}$, restricted operators $A_{0} = A\circ P_{\Tt_{x}^{(0)}}$, $L_{0} = L \circ P_{\Tt_{x}^{(0)}}$ and $\widehat{L}_{0} = L   - P_{\Tt_{z}^{(0)}} L P_{\pa{\Tt_{x}^{(0)}}^\bot}$. }

        \ENSURE $\xkp$\\[1mm]
        \WHILE{not converged}\vspace*{1mm}
            \STATE {\makecell[tl]{Algorithm updates:\\[1mm] 
            $\qquad \xkp  \leftarrow \texttt{algorithm}(y, A_{k}, L_{k}, \xk)$. \\[1mm]
            }}\\[1mm]
            
            \STATE {\makecell[tl]{Support updates: $ \betakp = -\frac{1}{\lambda}A^\top (A\xkp - y) .$ }}\\[1mm]
            
            \STATE ~~~~~\,{\tt Option I:} $\Ii_{x}^{(k+1)} = \Ii_{x}^{(k)} \mcup \Ba{ t \in \dbrack{\Nn} \mid \norm{\betakp_{G_t}} \geq w_t }$. \doubleslash{\cref{prop:lasso_certificate}} \\[1mm]
            
            \STATE ~~~~{\tt Option II:} { \makecell[tl]{$\ukp =  \widehat{L}_k \pa{ \widehat{L}_k^\top \widehat{L}_k }^{-1} {\betakp} , $ \\[1mm] $\Ii_{x}^{(k+1)} = \Ii_{x}^{(k)} \mcup \Ba{ t \in \dbrack{\Nn} \mid \norm{\uk_{J_t}} \geq 1 }$. \doubleslash{\cref{prop:ogn_certificate} } } } \\[1mm]
            
            \STATE Subspace and operators updates: $\Ee_{x}^{(k+1)} = \dbrack{n} \setminus \Pa{ \mcup_{t \in (\Ii_{x}^{(k+1)})^c} G_t }$\\[1mm]\makecell[tl]{
            $
            \begin{aligned}
                \qquad A_{k+1} &= A\circ P_{\Tt_{x}^{(k+1)}} \qandq
                L_{k+1} = L \circ P_{\Tt_{x}^{(k+1)}} , \\
                \widehat{L}_{k+1} &= L   - P_{\Tt_{z}^{(k+1)}} L P_{\pa{\Tt_{x}^{(k+1)}}^\bot}. 
            \end{aligned}
            $
            }\\[1mm]
        \ENDWHILE
    \end{algorithmic}
\end{algorithm}

\begin{remark}[Initialization]
   We adopt the correlation test in \cite{yuan2025adaptive} to initialize $\Ii^{(0)}_x$. In particular, for group $t\in \dbrack{\Nn}$, one can compute the score $\frac{\| X_{G_t}^\top y \| }{\|X_{G_t} \|_F \|y \|}$, and select the groups with highest scores as the initial $\Ii^{(0)}_x$. According to \cite{wang2013lasso}, a larger score indicates higher probability for this group being nonzero.
\end{remark}

\begin{remark}[The $\Ii^{(k+1)}_x$ update]
For the update of $\Ii^{(k+1)}_x$, in practice, instead of adding all groups satisfying $\|\beta^{(k)}_{G_t}\| \geq w_t$ or $\norm{u^{(k)}_{J_t}}\geq 1$ into $\Ii^{(k)}_x$, we choose incremental update by adding only the groups with largest values. 
\end{remark}

\begin{remark}[Computational overhead]\label{remark:overhead}
Since $\widehat{L}_k^\top \widehat{L}_k$ and $P_{\Tt_{x}^{(k+1)}}, P_{\Tt_{z}^{(k+1)}}$ are diagonal matrices (for $\widehat{L}_k^\top \widehat{L}_k$, see \cref{lem:hatL_full_rank}), the computation complexity of updating $\ukp, A_{k+1}, L_{k+1}$ and $\widehat{L}_{k+1}$ are negligible. 
While for computing $\betakp$, the complexity is $O(mn)$ which is the primary overhead caused by AdaDROPS. 
\end{remark}

\begin{remark}[Convergence of $\Ii_{x}^{(k)}$]
\label{fact:stable}
    Since $|\Ii_{x}^{(k)}|$ is monotonically increasing and upper bounded by $\Nn$, hence it converges to a fix set denoted as $\ol{\Ii}_{x}$.
\end{remark}

Given $\ol{\Ii}_{x} \subset \dbrack{\Nn}$, 
define $\ol{\Ee}_{x}, \ol{\Tt}_{x}$ according to \cref{eq:support_Ex}.
Define 
$
\ol{A} = A \circ P_{\ol{\Tt}_{x}}
$ and $
\ol{L} = L \circ P_{\ol{\Tt}_{x}} 
$, the restriction of problem \cref{eq:main_problem} to $\ol{\Tt}_{x}$ is
\begin{equation}
\label{eq:restricted_problem}
    \min_{x \in \RR^{n}}  \bBa{ \ol{\Phi}(x) \eqdef
    \sfrac{1}{2\lambda} \norm{\ol{A} x - y}^2 +  \norm{\ol{L} x}_{1,2} } .
\end{equation}
Since \cref{eq:restricted_problem} focuses on the entries of $x$ restricted to $\ol{\Ee}_x$, for the rest of the discussion we impose that $x_{\ol{\Ee}_x^c}=0$. 
Assume {\tt algorithm} is convergent, then $x^{(k)}$ converges to a solution, say $\tilde{x} \in \RR^n$, of the restricted problem \cref{eq:restricted_problem}. Denote
\begin{equation}
\label{eq:def_u}
\tilde{\beta} \eqdef -\sfrac{1}{\lambda} A^\top (A\tilde x-y),\quad
\ol{L}  \eqdef L  - P_{\ol{\Tt}_{z}} L P_{\ol{\Tt}_{x}^\bot},\quad 
\ol{u} \eqdef  \ol{L} \Pa{ \ol{L}^\top \ol{L} }^{-1} \tilde{\beta} .
\end{equation}
We have the following result regarding the optimality of AdaDROPS.

\begin{lemma}
\label{lem:unorm_leq1}
    For the vector $\ol{u}$, there holds $\norm{\ol{u}_{J_t}} \leq 1$ for any $t\in\ol{\Ii}_{x}^c$.
\end{lemma}
\begin{proof}
For any $t \in \ol{\Ii}_{x}^c$, the updating rule implies $\norm{u_{J_t}^{(k)}} < 1$ for {\tt Option II} and $\norm{\tilde{\beta}_{G_t}^{(k)}/w_t} < 1$ for {\tt Option I}. Following the same proof as in \cref{prop:link_two}, we derive the inequality  $\norm{u_{J_t}^{(k)}} \leq \norm{\tilde{\beta}_{G_t}^{(k)}/w_t} < 1$.
Note that after finite iterations, $u_{J_t}^{(k+1)} = \ol{L} \pa{\ol{L}^\top \ol{L}}^{-1} \beta^{(k+1)} \to \ol{u}_{J_t}$, we obtain
$\norm{\ol{u}_{J_t}} \leq 1$ for both options.
\end{proof}

Rely on \cref{lem:unorm_leq1}, next we show that the solution of the restricted problem \cref{eq:restricted_problem} is also a solution of the original problem \cref{eq:main_problem}.

\begin{proposition}
\label{thm:opt_AS}
AdaDROPS \cref{algo:adadrops} is convergent, the output $\tilde{x}$ is a solution of the original problem \cref{eq:main_problem}. 
\end{proposition}
\begin{proof}
Since $\tilde{x}$ is a solution of the restricted problem \cref{eq:restricted_problem} and $\tilde{x}_{\ol{\Ee}_{x}^c}=0$ by our configuration, its optimality condition yields
\[
\begin{aligned}
- \sfrac{1}{\lambda} \ol{A}^\top ( \ol{A} \tilde{x} - y)  
=- \sfrac{1}{\lambda} \ol{A}^\top ( A \tilde{x} - y)  
\in \ol{L}^\top \partial \norm{\ol{L}\tilde{x}}_{1,2} = \ol{L}^\top \partial \norm{L\tilde{x}}_{1,2}  .
\end{aligned}
\]
Hence, there exists ${{u}} \in \partial \norm{L\tilde{x}}_{1,2} $ such that $- \frac{1}{\lambda} \ol{A}^\top ( A \tilde{x} - y)  = \ol{L}^\top {{u}} $. By the definition of $\tilde{\beta}$ \cref{eq:def_u}, there holds $P_{\Tt_x} \tilde{\beta} = P_{\Tt_x} L^\top {{u}}$.
To show $\tilde{x}$ is a solution of \cref{eq:main_problem}, by the optimality condition, it means 
\begin{equation}\label{eq:tilde_beta}
\tilde{\beta} =  - \sfrac{1}{\lambda}  A^\top ( A \tilde{x} - y)  
\in  L^\top \partial \norm{L\tilde{x}}_{1,2}  .
\end{equation}
To prove \cref{eq:tilde_beta}, it suffices to constructing a dual variable $\tilde{u} \in \RR^p$ such that 
$\tilde{u} \in \partial \norm{L\tilde{x}}_{1,2}$ and $
L^\top \tilde{u} = \tilde{\beta}$. 
%
Given $\ol{\Ii}_x$, denote $\ol{\Tt}_L$ according to \cref{eq:support_EL}.
We choose $\tilde{u}$ by
\begin{equation}\label{eq:tilde_u}
\tilde{u}_{\ol{\Ee}_{z}} =  \pa{ P_{\ol{\Tt}_{L}} {{u}} }_{\ol{\Ee}_{z}}  
    \qandq
    \tilde{u}_{\ol{\Ee}_{z}^c} = \ol{u}_{\ol{\Ee}_{z}^c} ,
\end{equation}
and verify that it satisfies the requirement. 
Note that there holds
\begin{align*}
P_{\ol{\Tt}_{x}^\bot} L^\top P_{\ol{\Tt}_{z}} \tilde{u} 
&=  
P_{\ol{\Tt}_{x}^\bot} L^\top P_{\ol{\Tt}_{z}} P_{\ol\Tt_{L}} {{u}} \doubleslash{\cref{eq:tilde_u}}  \\
&=  
P_{\ol{\Tt}_{x}^\bot} L^\top P_{\ol\Tt_{L}} {{u}} \doubleslash{$\ol{\Tt}_{L} \subseteq \ol{\Tt}_{z}$, \cref{lemma:proj} (a)} \\
&= 0 . \doubleslash{\cref{lemma:proj} (b)} 
\end{align*} 
%
The proof follows the same argument as that of \cref{prop:ogn_certificate}, yielding $L^\top \tilde{u} = \tilde{\beta}$.

Next we show $\tilde{u} \in \partial \norm{L\tilde{x}}_{1,2}$, which consists of two parts due to the definition of $\tilde{u}$
\begin{itemize}
    \item[i).]  We have $P_{\ol\Tt_L}L\tilde x = 0$ since $\tilde x \in \ol{\Tt}_x$. Recall that ${{u}} \in \partial \norm{L\tilde x}_{1,2}$, hence $P_{\ol\Tt_L}{{u}} \in \partial\norm{L\tilde x}_{1,2}$. 
    As $\ol\Tt_L \subseteq \ol\Tt_{z}$ (\cref{lemma:proj} \BLUE{(a)}), by the definition of $\tilde{u}$ we have
    \[
    P_{\ol\Tt_{z}} \tilde{u}
    = P_{\ol\Tt_{L}} {{u}}
    = P_{\ol\Tt_{z}} P_{\ol\Tt_{L}} {{u}} \in  P_{\ol\Tt_{z}} \partial \norm{L\tilde x}_{1,2} .
    \]
    \item[ii).] For $t \in \ol{\Ii}_{x}^c$, $\norm{\tilde{u}_{J_t}} \leq 1$ by \cref{lem:unorm_leq1}. 
\end{itemize}
In summary, we get $\tilde{u} \in \partial \norm{L\tilde{x}}_{1,2}$ which concludes the proof.  
\end{proof}

\subsection{Related work}\label{sec:related_work} 
In this part, we provide a brief summary of existing dimension reduction techniques for overlapping group sparsity and highlight the difference of our approach. More precisely, we discuss the adaptive dimension reduction for sparse optimization \cite{yuan2025adaptive,jenatton2011structured}, efficient proximal operator computation of overlapping group sparsity (``FoGLASSO'') \cite{yuan2011efficient}, and screening rules for overlapping group sparsity \cite{lee2014screening}. 

\paragraph{Adaptive dimension reduction \cite{yuan2025adaptive, jenatton2011structured}}
From the optimality condition \cref{eq:opt-cnd}, given any $x\in\RR^n$, define  
\begin{equation}
\label{eq:KKT_lasso}
   \Rr(x) \eqdef  x - \mathrm{prox}_{ R }\Pa{ x - \sfrac{1}{\lambda} A^\top (Ax -y) }  ,
\end{equation}
which is called the {\it KKT residual}. 
Based on $\Rr(x)$, \cite{yuan2025adaptive} proposed an adaptive sieving (AS) scheme to achieve dimension reduction: at each iteration, entries satisfying $|R(x)_i| > 0$ are added to an active set, and a reduced problem is solved over this set. However, AS requires the explicit calculation of the proximal operator $\mathrm{prox}_{ R }$, which excludes the case of overlapping group norm. 
In \cite{jenatton2011structured}, the authors propose an active-set strategy for structured overlapping group sparsity, which updates the approximate support by checking the duality gap. However, their objective differs from \cref{eq:main_problem} as they consider squared group norm.

\paragraph{The FoGLASSO \cite{yuan2011efficient}}

FoGLASSO leverages \cref{def:lasso-cert} to identify the support in the special case $A = \Id$. Specifically, they show that if $\norm{y_{G_i}} < \lambda w_i$, then $x^{\star}_{G_i} = 0$. 
This property is plugged into the computation of the proximal operator of the group norm, achieving dimension reduction and improved efficiency. 
However, it is unclear if their zero group identification rule extends to $A \neq \Id$.

\paragraph{Safe screening \cite{lee2014screening}}
Recall \cref{eq:pre_certificate}, group $t \in \dbrack{\Nn}$ is zero if $\norm{\usol_{J_t}} < 1$, which is
\begin{equation*}
    \norm{\usol_{J_t}}^2 
    = \msum_{i \in G_t} (\betasol_i - \ssum_{k\in \mathscr{K}_i\setminus \ba{k_i}, j = \phi(k) } w_j {{\usol_k}})^2 /w_t < 1.
\end{equation*} 
To estimate the summand, one needs to know the value of $\usol_k$, our approach is splitting the entries in $\usol_{J_t}$ into two parts: the non-zero part of $\usol_k$ has the same sign as $\betasol_i$, yielding the inequality. 
In \cite{lee2014screening}, for zero groups the authors propose to compute the minimum of the summand above as the subdifferential is set-valued, leading to the following condition
\begin{equation*}
    \min_{u\in\RR^p,\forall s \in \dbrack{\Nn} \setminus \ba{t}, \norm{u_{J_s}} \leq 1} \sqrt{\ssum_{i \in G_t}(\betasol_i - \ssum_{k\in \mathscr{K}_i\setminus \ba{k_i}, j = \phi(k)  }w_j u_k)^2} < w_t  \Longrightarrow x_{G_t}^{\star} = 0.
\end{equation*}
Note that $\norm{\betasol_{G_t}}$ is an upper bound of the minimization problem as it corresponds to taking $u_k=0$. 
Consequently, a static safe screening rule is designed in \cite{lee2014screening}. 
We remark that their result mainly focuses on reducing the influence caused by the groups completely inclusive of another group (mainly addressing the case of sparse overlapping group LASSO), while neglecting other types of overlapping. Specifically, they seek an upper bound for
\begin{equation*}
    \min_{u\in\RR^p,\forall s \in \{ s\in \dbrack{\Nn} \setminus \ba{t}, G_s \cap G_t \neq \varnothing, G_s \subseteq G_t \}, \norm{u_{J_s}} \leq 1} \sqrt{\ssum_{i \in G_t}(\betasol_i - \ssum_{k\in \mathscr{K}_i\setminus \ba{k_i}, j = \phi(k) }w_j u_k)^2}.
\end{equation*}
In contrast, our approach works for the general overlapping cases.

\section{Applications}
\label{sec:applications}

In this section, we discuss the choices of {\tt algorithm}. 
As the proximal operator of overlapping group norm has no closed form expression, choices to alleviate this difficulty include Primal--Dual splitting method \cite{chambolle2011first}, alternating direction method of multipliers \cite{gabay1983chapter}, and variable projection \cite{poon2023smooth} which is a recently developed versatile numerical scheme for sparse optimization. 
In what follows we first provide brief introduction of these methods, highlighting their computational complexity, then demonstrate how to combine them with the proposed AdaDROPS.

\subsection{Primal-Dual splitting method}

Given a function $R:\RR^n \to \RR$, its convex conjugate is defined as $R^*(\psi) = \sup_{x} x^\top \psi - R(x)$. Moreover, when $R$ is proper closed and convex, its biconjugate equals itself, i.e. $R(x) = \sup_{\psi} \psi^\top x - R^*(\psi)$. Plugging this into~\cref{eq:main_problem} we obtain the following saddle-point problem 
\begin{equation}\label{eq:saddle-point}
\min_{x\in\RR^n}\max_{\psi\in\RR^p} \sfrac{1}{2\lambda}\norm{Ax-y}^2 + \iprod{Lx}{\psi} - \iota_{\Omega}(\psi) ,
\end{equation}
where $\Omega = \Ba{\psi\in\RR^p \mid \norm{\psi_{J_i}}_2 \leq 1,~ i\in\dbrack{\Nn}}$ and $\iota_{\Omega}(\cdot)$ is the indicator function of $\Omega$.

In the literature, a rich class of Primal-Dual splitting methods is developed to solve the saddle-point problem, see for instance \cite{esser2010general,chambolle2011first,vu2011splitting} and the references therein.

\begin{algorithm}[H]
    \renewcommand{\algorithmicrequire}{\textbf{Input:}}
	\renewcommand{\algorithmicensure}{\textbf{Output:}}
    \small
    \caption{Primal-Dual splitting method \cite{chambolle2011first}}
    \begin{algorithmic}[1]
    \label{algo:primal-dual}
        \REQUIRE  initial point $x^{(0)}\in\RR^{n}, \psi^{(0)}\in\RR^{p}$, stepsize $\sigma, \tau>0$ such that $\sigma\tau \norm{L}^2 < 1$.
        \ENSURE $ x^{(k)}$
        \WHILE{not converged}\vspace{1mm}
            \STATE $\xkp = \pa{\lambda \Id_n + \sigma A^\top  A }^{-1}\pa{\lambda \xk -\lambda  \sigma L^\top \psi^{(k)} + \sigma A^\top y } $. \doubleslash{primal update}\vspace{1mm}
            
            \STATE $\bar{x}^{(k+1)} = 2\xkp -\xk $; \doubleslash{extrapolation}\vspace{1mm}
            
            \STATE $\psi^{(k+1)} = P_{\Omega} \pa{ \psi^{(k)} + \tau L \bar{x}^{(k+1)} }$. \doubleslash{dual update}
        \ENDWHILE
    \end{algorithmic}
\end{algorithm}

\begin{remark}[Convergence and complexity]
The above iteration can be written as an instance of proximal point algorithm, hence its convergence is guaranteed \cite{vu2011splitting}. 
As $L$ is sparse, the complexity of the matrix-vector products involving $L$ is $O(p)$ which is negligible. For the primal update, we need to compute a matrix inversion and matrix-vector product. Note that when $m < n$, the matrix inversion in \cref{algo:primal-dual} {\tt Line 2} can be computed by inverting the following smaller $m \times m$ matrix according to the Sherman--Morrison--Woodbury formula \cite{golub2013matrix},
\begin{equation*}
    (\lambda \Id_n + \sigma A^\top A)^{-1} = \sfrac{1}{\lambda} \Id_n - \sfrac{1}{\lambda} A^\top \Pa{\sfrac{\lambda}{\sigma} \Id_m + A A^\top}^{-1} A.
\end{equation*}
The Cholesky decomposition of $\lambda \Id + \sigma A^\top  A$ and $\sfrac{\lambda}{\sigma} \Id_m + A A^\top$ can be precomputed. Consequently, the computational cost of \cref{algo:primal-dual} is $O(mn^2 + n^3) +O( kn^2)$ when $m\geq n$, where $k$ denotes the number of iteration needed for convergence. When $m<n$, the complexity shifts to $O(m^2 n+m^3 )+O( kmn)$. Alternatively, the linear system can be solved by preconditioned conjugate gradient method, where the complexity depends on the condition number of the matrix.
\end{remark}

\subsection{Alternating direction method of multipliers}

Another widely adopted approach to handle overlapping group norm is introducing auxiliary variables and solving with method of multipliers. More precisely, problem \cref{eq:main_problem} is equivalent to the following constrained problem
\[
\begin{aligned}
\min_{x\in\RR^n, z\in\RR^p} ~& \sfrac{1}{2\lambda}\norm{Ax-y}^2 + \norm{z}_{1,2} \\
{\rm such~that}~& z = L x ,
\end{aligned}
\]
whose augmented Lagrange function reads
\begin{equation*}
\mathcal{L}_{\tau}(x,z, \psi)
\eqdef 
\sfrac{1}{2\lambda} \norm{Ax - y}_2^2 + \norm{z}_{1,2} + \langle\psi, z - Lx \rangle + \sfrac{\tau}{2} \norm{z - Lx}_2^2 .
\end{equation*}
Alternating direction method of multipliers (ADMM) \cite{glowinski1975approximation,gabay1983chapter} applies Gauss--Seidel updating rules to the three variables, leading to \cref{algo:admm}.

\begin{algorithm}[H]
    \renewcommand{\algorithmicrequire}{\textbf{Input:}}
	\renewcommand{\algorithmicensure}{\textbf{Output:}}
    \small
    \caption{Alternating direction method of multipliers \cite{gabay1983chapter}}
    \begin{algorithmic}[1]
    \label{algo:admm}
        \REQUIRE  initial point $x^{(0)}\in\RR^{n}, \psi^{(0)}\in\RR^{p}$, $\tau>0$.
        \ENSURE $ x^{(k)}$
        \WHILE{not converged}\vspace{1mm}
            \STATE $\xkp = \pa{ A^\top A + \lambda\tau L^\top L }^{-1} \Pa{ A^\top y + \lambda L^\top (\psik + \tau \zk) } $. \doubleslash{primal update}\vspace{1mm}
            
            \STATE $\zkp = \mathrm{prox}_{\frac{1}{\tau} \norm{\cdot}_{1,2}} \pa{ L \xkp - \sfrac{1}{\tau} \psik }$; \doubleslash{primal update}\vspace{1mm}
            
            \STATE $\psikp = \psik + \tau \pa{ \zkp - L\xkp }$. \doubleslash{dual update}
        \ENDWHILE
    \end{algorithmic}
\end{algorithm}

\begin{remark}[Convergence and complexity]
Since $L$ has full column-rank, \cref{algo:admm} is sequence convergent \cite{eckstein1992douglas}. 
In terms of computational complexity, ADMM has the same as that of \cref{algo:primal-dual}.
\end{remark}

\subsection{Variable projection}\label{sec:varpro}

The last method to introduce is the variable projection (VarPro) \cite{poon2021smooth,poon2023smooth}, which is developed based on Hadamard overparameterization.

\subsubsection{Hadamard overparameterization}\label{sec:hadamard}

Continue from \cref{sec:overlapping_group_norm},  
let $u\in\RR^p, v\in\RR^{\Nn}$ and $\Jj$ be an $\Nn$-group partition of $\dbrack{p}$. 
The Hadamard product of $u, v$, denoted by $u\odot_\Jj v$, is a vector in $\RR^p$ defined by
\[
u \odot_\Jj v = \pa{ u_{J_i} \times {v}_{i} }_{i=1}^{\Nn} . 
\]
We have the following variational form of overlapping group norm. 

\begin{definition}
The Hadamard overparameterization of the group norm is 
\begin{equation}\label{eq:Hadamard-Lx}
    \msum_{i\in\dbrack{\Nn}} w_i \norm{x_{G_i}}
    = \min\limits_{u\in\RR^p,\, v\in\RR^{\Nn}} \left\{ \sfrac{1}{2} \norm{u}^2 + \sfrac{1}{2} \norm{v}^2 ,~~ L x = u\odot_\Jj v \right\}  . 
\end{equation}
For each group $J_i\in\Jj$, the minimum is obtained for 
\[
u_{J_i} = {L_{G_i}x} /{\sqrt{\norm{L_{G_i}x}}}
\qandq
v_i = \sqrt{ \norm{ L_{G_i} x } }  . 
\]
\end{definition}

With overparameterization, the original problem \cref{eq:main_problem} can be rewritten into the  constrained form
\begin{equation}
\label{eq:Guv}
\min_{v\in\RR^{\Nn},\, u\in\RR^{p},\, x\in\RR^{n}} \left\{ \sfrac{1}{2} \norm{u}^2 + \sfrac{1}{2} \norm{v}^2 + \sfrac{1}{2\lambda} \norm{Ax-y}_2^2,~~ Lx = u\odot_{\Jj} v\right\} ,
\end{equation}
which now is smooth but nonconvex due to the bilinear term.

\subsubsection{Variable projection}

Though smooth, \cref{eq:Guv} is not easy to solve as it is bilevel and constrained. 
In \cite{poon2023smooth}, based on the variable projection (VarPro) technique \cite{golub1973differentiation,golub2003separable}, \cref{eq:Guv} is further transformed to a bilevel problem
\begin{equation}
\label{eq:varpro:projected}
\begin{aligned}
    \min_{v\in\RR^{\Nn}}~& f(v), \\
    \text{ where }~& f(v) \triangleq \min_{u\in\RR^{p},\, x\in\RR^{n}}\bBa{ \sfrac{1}{2} \norm{u}^2 + \sfrac{1}{2} \norm{v}^2 +  \sfrac{1}{2\lambda} \norm{Ax-y}_2^2,~~ Lx = u \odot_{\Jj} v } ,
    \end{aligned}
\end{equation}
where $f(v)$ is called the \emph{projected function}, and $v$ is the \emph{projected variable}.
Note that the lower-level problem is convex in $x,u$ with quadratic objective and linear constraint, we can consider its dual problem \cite{poon2023smooth}.

\begin{lemma}[{\cite{poon2023smooth}}]
    Problem \cref{eq:varpro:projected} is equivalent to the following bilevel problem
    \begin{equation}
    \label{eq:min_f}
    \begin{aligned}
        \min_{v \in \RR^{\Nn}}~ f(v), \quad
        \mathrm{where}~& f(v) = \max\limits_{\alpha\in\RR^m,\, \xi\in \mathbb{R}^p} \sfrac{1}{2} \norm{v}^2 - \sfrac{1}{2} {\norm{\xi\odot_{\Jj} {v}}^2} - \sfrac{\lambda}{2} \norm{\alpha}^2 - \iprod{\alpha}{y} , \\
        &\qquad  \mathrm{such~that}\quad L^\top \xi + A^\top \alpha = 0 . 
        \end{aligned}
    \end{equation}
    Given $v\in\RR^{\Nn}$, let $(\alpha,\xi)$ be a solution of the lower-level problem, then 
    \begin{itemize}
        \item There exists $x\in\RR^n$ such that
    \begin{equation}
    \label{eq:three}
        \lambda \alpha = Ax-y ,\quad {Lx = \xi\odot_{\Jj} {v}^2} \qandq L^\top \xi + A^\top \alpha = 0.
    \end{equation}

        \item The gradient of the upper-level projected function reads
            \[
                \nabla f(v) = v- v \odot (\norm{\xi_{J_i}}^2)_{i=1}^{\Nn}  . 
            \]
    \end{itemize}
    
\end{lemma}

The above result implies we can apply descent methods, e.g. gradient descent or quasi-Newton method, to solve \cref{eq:min_f}. 
To solve the lower-level problem \cref{eq:three}, denote $D_v = \diag\pa{ \bm{1}_{p} \odot_{\Jj} (1/ v^2) }  \in \RR^{p\times p}$, and $W = L^\top D_v  L  \in \RR^{n\times n} $ which is diagonal.
Then we have two choices solving \cref{eq:three}
\begin{itemize}
    \item If $n\leq m$, 
\begin{equation}
\label{eq:ogl:1}
\pa{ A^\top A + \lambda W}  x = A^\top y ,\quad  \alpha = \pa{ Ax-y } /\lambda .
\end{equation}

\item If $m< n$, by the Sherman--Morrison--Woodbury formula \cite{golub2013matrix},
\begin{equation}
\label{eq:ogl:2}
\pa{\lambda \Id_m +AW^{-1} A^\top}  \alpha = -y ,\quad x = -W^{-1}A^\top \alpha.
\end{equation}

\end{itemize}
For both cases, we have $\xi = D_v Lx$. 
 
Assemble the above elements, an example of gradient descent solving \cref{eq:varpro:projected} is provided in \cref{algo:standard:varpro}. 
\begin{remark}[Convergence and complexity]\label{remark:complexity-varpro}
The convergence of VarPro is guaranteed according to \cite{poon2023smooth}. 
The complexity lies in again the matrix inversion and matrix-vector/matrix product. 
\end{remark}

\begin{algorithm}[H]
    \renewcommand{\algorithmicrequire}{\textbf{Input:}}
	\renewcommand{\algorithmicensure}{\textbf{Output:}}
    \small
    \caption{Gradient descent based VarPro \cite{poon2023smooth}}
    \begin{algorithmic}[1]
    \label{algo:standard:varpro}
        \REQUIRE  initial point $v^{(0)}\in\RR^{\Nn}$.
        \ENSURE $ x^{(k)}$
        \WHILE{not converged}
            \STATE \makecell[tl]{Lower-level updates\\[1mm]
            $\quad D_k = \diag\Pa{ \bm{1}_{p} \odot_{\Jj} (1/ (v^{(k)})^2) } , ~W_k = L^\top D_k L$. \\[1mm]
            \qquad\quad {\bf Case $n \leq m$:} $x^{(k)} = \pa{ A^\top A + \lambda W_k }^{-1} A^\top y$; \\[1mm]
            \qquad\quad {\bf Case $m < n$:} $ \alpha^{(k)} = -\pa{ \lambda \Id_m +AW_k^{-1} A^\top }^{-1}y ,~ x^{(k)} = -W_k^{-1}A^\top\alpha^{(k)}$; \\[1mm]
            $\quad \xi^{(k)} = D_kL x^{(k)}$; \vspace{1mm} 
            }

            \STATE \makecell[tl]{
            Upper-level updates\\[1mm]
            $
            \begin{aligned}
                \quad g^{(k)} &= v^{(k)} - v^{(k)} \odot  (\norm{\xi^{(k)}_{J_i}}^2)_{i=1}^{\Nn} \doubleslash{Gradient $\nabla f(\vk)$} , \\
                \quad v^{(k+1)} &= v^{(k)} - \gamma^{(k)}  g^{(k)} . \doubleslash{$\gamma^{(k)}$ can be determined by line search.}
            \end{aligned}
            $
            }\\[1mm]

        \ENDWHILE
    \end{algorithmic}
\end{algorithm}

\subsection{Complexity reduction via AdaDROPS}

The above algorithms can be easily plugged into \cref{algo:adadrops}, 
instead of working on the full dimension problem \cref{eq:main_problem}, in each step of AdaDROPS, the {\tt algorithm} solves \cref{eq:restricted_problem_k}, which essentially boils down to the update of $\xkp$. 
In the following, we use ADMM \cref{algo:admm} to demonstrate how to compute $\xkp$.

Denote $r^{(k+1)} =  A_{k}^\top y + \lambda L_{k}^\top (\psik + \tau \zk)$. Apparently, $r^{(k+1)} \in \Tt_x^{(k)}$ owing to the definition of $A_{k}, L_{k}$. 
The update of $x^{(k+1)}$ in \cref{algo:admm} is equivalent to
\begin{equation}
\label{eq:ADMM_x}
    (A_k^\top A_k + \lambda \tau L^\top_k L_k) x^{(k+1)} = r^{(k+1)} ,
\end{equation}
which means we need to invert $(A_k^\top A_k + \lambda \tau L^\top_k L_k)$ over $\Tt_{x}^{(k)}$. 
Denote $A_{\Ee_{x}}, L_{\Ee_{x}}$ and $\xkp_{\Ee_x}, r^{(k+1)}_{\Ee_x}$ the restriction of $A,L$ (columns) and $\xkp, r^{(k+1)}$ on $\Ee_{x}^{(k)}$, then
\begin{equation*}
  x^{(k+1)}_{\Ee_{x}} = \Pa{A_{\Ee_{x}}^\top A_{\Ee_{x}} + \lambda \tau L_{{\Ee_{x}}}^\top L_{\Ee_{x}}}^{-1} r^{(k+1)}_{\Ee_x} . 
\end{equation*}
Letting $x^{(k+1)}_{\Ee_{x}^c} = 0$ yields the whole $\xkp$. 
If $m < |\Ee_x|$, then Sherman--Morrison--Woodbury formula can be applied: let $H = (L_{\Ee_x}^\top L_{\Ee_x})^{-1}$
\begin{align*}
    x^{(k+1)}_{\Ee_{x}} 
    =& \sfrac{1}{\lambda \tau}H  \Pa{ r^{(k+1)}_{\Ee_x} - A_{\Ee_x}^\top (\lambda \tau \Id_m + A_{\Ee_x}H A_{\Ee_x}^\top )^{-1} A_{\Ee_x}H r^{(k+1)}_{\Ee_x} } .
\end{align*}
See \cref{algo:admm-adadrops} below for the whole procedure of ADMM with AdaDROPS, for the case of $|\Ee_x| \leq m$.

\begin{algorithm}[H]
    \renewcommand{\algorithmicrequire}{\textbf{Input:}}
	\renewcommand{\algorithmicensure}{\textbf{Output:}}
    \small
    \caption{AdaDROPS for ADMM}
    \begin{algorithmic}[1]
    \label{algo:admm-adadrops}
        \REQUIRE  {initial estimated support $\Ii_{x}^{(0)}$, restricted operators  $A_{0} = A\circ P_{\Tt_{x}^{(0)}}$, $L_{0} = L \circ P_{\Tt_{x}^{(0)}}$ and $\widehat{L}_{0} = L   - P_{\Tt_{z}^{(0)}} L P_{\pa{\Tt_{x}^{(0)}}^\bot}$. }
        \ENSURE $\xkp$\\[1mm]
        \WHILE{not converged}\vspace*{1mm}
            \STATE {\makecell[tl]{ADMM updates:\\[1mm] 
            $
            \begin{aligned}
            \qquad r^{(k+1)} &=  A_{k}^\top y + \lambda L_{k}^\top (\psik + \tau \zk),\\
            \qquad \xkp_{\Ee_x} &= \Pa{A_{\Ee_{x}}^\top A_{\Ee_{x}} + \lambda \tau L_{{\Ee_{x}}}^\top L_{\Ee_{x}}}^{-1} r^{(k+1)}_{\Ee_x} ,\quad x^{(k+1)}_{\Ee_x^c} = 0, \\
            \qquad \zkp &= \mathrm{prox}_{\frac{1}{\tau} \norm{\cdot}_{1,2}} \Pa{ L_{k} \xkp - \sfrac{1}{\tau} \psik } , \\
            \qquad \psikp &= \psik + \tau \pa{ \zkp - L_{k}\xkp } ,
            \end{aligned}
            $}}\\[1mm]

            \STATE {\makecell[tl]{Support updates: \cref{algo:adadrops} {\tt Line 3 - Line 5}. 
            }
            }\\[1mm]

            \STATE Subspace and operators updates: \cref{algo:adadrops} {\tt Line 6} \\[1mm]
        \ENDWHILE
    \end{algorithmic}
\end{algorithm}

As mentioned in \cref{remark:overhead} the computational overhead of AdaDROPS is negligible, hence we mainly focus on the computational complexity reduction. 
The main computation complexity of the three algorithms lies in the update of $\xkp$, which involves matrix inversion and matrix-vector/matrix multiplication. 
When both $m$ and $n$ are large, precomputing the linear system and its Cholesky decomposition may lead to memory issues. Hence, we analyze the complexity reduction of AdaDROPS in two scenarios: with ({\tt w.}) and without ({\tt w.o.}) precomputing the inversion/decomposition.
The comparison is summarized in \cref{tab:admm_complexity}.

\begin{table}[!htb]
\centering
\small
\caption{Complexity comparison between ADMM  and ADMM+AdaDROPS in $k$-th iteration, let $\kappa_k = |\Ee_x^{(k)}|$ be the problem dimension in $k$-th iteration.}\vspace*{-1ex}
\tabulinesep=.75mm
\begin{tabu}{|c|c|c|c|}
\hline
case                                           &                 & ADMM          & ADMM+AdaDROPS                            \\ \hline
\multirow{2}{*}{$m \geq n \geq \kappa_k$} & {\tt w.}     & $O(n^2)$      & $O(\kappa_k^2)$                     \\ \cline{2-4} 
                                               & {\tt w.o.} & $O(n^2(m+n))$ & $O(\kappa_k^2(m + \kappa_k)) $ \\ \hline
\multirow{2}{*}{$n > m \geq \kappa_k$}    & {\tt w.}     & $O(mn)$       & $O(\kappa_k^2)$                     \\ \cline{2-4} 
                                               & {\tt w.o.} & $O(m^2(m+n))$ & $O(\kappa_k^2(m + \kappa_k)) $ \\ \hline
\multirow{2}{*}{$n \geq \kappa_k\geq m$}  & {\tt w.}     & $O(mn)$       & $O(m\kappa_k)$                      \\ \cline{2-4} 
                                               & {\tt w.o.} & $O(m^2(m+n))$ & $O(m^2(m+\kappa_k))$                \\ \hline
\end{tabu}
\label{tab:admm_complexity}
\end{table}

\begin{remark}
The above complexity reduction also holds for Primal--Dual method.
\end{remark}

\begin{table}[!htb]
\centering
\small
\caption{Complexity comparison between VarPro and VarPro+AdaDROPS in $k$-th iteration, let $\kappa_k = |\Ee_x^{(k)}|$ be the problem dimension in $k$-th iteration.}\vspace*{-1ex}
\tabulinesep=.75mm
\begin{tabu}{|c|c|c|c|}
\hline
case                                           &                 & VarPro          & VarPro+AdaDROPS                            \\ \hline
\multirow{2}{*}{$m \geq n \geq \kappa_k$} & {\tt w.}     & $O(n^3)$      & $O(\kappa_k^3)$                     \\ \cline{2-4} 
                                               & {\tt w.o.} & $O(n^2(m+n))$ & $O(\kappa_k^2(m + \kappa_k)) $ \\ \hline
\multirow{2}{*}{$n > m \geq \kappa_k$}    & {\tt w.}     & $O(m^2(m+n))$       & $O(\kappa_k^3)$                     \\ \cline{2-4} 
                                               & {\tt w.o.} & $O(m^2(m+n))$ & $O(\kappa_k^2(m + \kappa_k)) $ \\ \hline
\multirow{2}{*}{$n \geq \kappa_k\geq m$}  & {\tt w.}     & $O(m^2(m+n))$       & $O(m^2(m+\kappa_k))$                      \\ \cline{2-4} 
                                               & {\tt w.o.} & $O(m^2(m+n))$ & $O(m^2(m+\kappa_k))$                \\ \hline
\end{tabu}
\label{tab:varpro_complexity}
\end{table}

For VarPro, each iteration involves different linear systems due to the matrix $D_k$. Consequently, while the matrix $A^\top A$ can be precomputed, the Cholesky decomposition of \cref{eq:ogl:1} and \cref{eq:ogl:2} cannot be reused across iterations. 
A detailed breakdown of the computational complexity in the $k$-th iteration is provided in \cref{tab:varpro_complexity}.

The above discussions imply that eventually the computational complexity depends on $\abs{\ol{\Ee}_{x}}$. 
This means the practical acceleration offered by AdaDROPS not only depends on the sparsity of the solution, more importantly depends on how accurate the estimated support $\ol{\Ee}_{x}$ is. 
In the next section we shall see that {\tt Option II} in general offers a much better estimation quality.

\section{Numerical experiments}
\label{sec:numerics}
To evaluate the efficiency of AdaDROPS, we consider both overlapping and nonoverlapping sparse optimization problems. 
We use the suffix ``Acc-$\beta$/$u$" to refer algorithms combined with AdaDROPS with {\tt Option I or II}. 
For VarPro \cref{algo:standard:varpro}, L-BFGS method \footnote{\url{https://github.com/stephenbeckr/L-BFGS-B-C}} \cite{byrd1995limited} is adopted to solve the upper-level problem. 
For linear system solving, when the dimension is small, a direct Cholesky solver is applied, while for large dimension the preconditioned conjugate gradient (PCG) is adopted. 
All experiments are performed on an Apple MacBook Pro (2021) equipped with an Apple M1 Max and 64GB of unified memory. All code is implemented in MATLAB R2025b, with source code of our methods available at the GitHub repository\footnote{\url{https://github.com/TTony2019/AdaDROPS}}.

\subsection{Overlapping group sparsity}\label{sec:numerics:overlapping}

For overlapping group sparsity, two problems are considered: the classic overlapping group LASSO on datasets from LIBSVM~\cite{chang2011libsvm}, and the synthesis approach wavelet image processing problem \cite{elad2007analysis}.

\subsubsection{LIBSVM datasets}

In this example, we consider ADMM/VarPro with and without AdaDROPS, and compare with SLEP\footnote{\url{https://yelabs.net/software/SLEP/}} method proposed in \cite{Liu:2009:SLEP:manual, yuan2011efficient}. For this experiment, three datasets {\tt gisette}, {\tt E2006.test} and {\tt E2006.train}, are considered. The configurations of the problems, including problem dimension ($m,n$), regularization parameter $\lambda$, number of total groups $\Nn$, group size ({\tt gs}), overlapping size ({\tt os}), cardinality of the support of optimal solution $\abs{\supp(\xsol)}$ and number of nonzero groups $\abs{\Ii_{\xsol}}$, for the three datasets are summarized in the \cref{tab:configuration}.

\begin{table}[!htb]
\centering
\small
\caption{Configurations of the overlapping group LASSO problem for the three datasets. Denote $\bar\lambda  = \max_{i \in \dbrack{\Nn}}\norm{A^\top_{G_i} y}/w_i >0$. 
}\vspace*{-1ex}
\tabulinesep=.75mm
\begin{tabu}{c|c|c|c|c|c|c|c}
\hline
 & $(m,n)$ & $\lambda$ & $\Nn$ & {\tt gs} & {\tt os} & $\abs{\supp(\xsol)}$ & $\abs{\Ii_{\xsol}}$ \\ \hline
{\tt gisette} & $(6000,5000)$ & $\bar\lambda/10$ & $1000$ & $7$ & $2$ & $157$ &  $48$ \\ \hline
{\tt E2006.test} & $(3308,150358)$ & $\bar\lambda/10^6$ & $15036$ & $50$ & $40$ & $24$ & $11$ \\ \hline
{\tt E2006.train} & $(16087,150360)$ & $\bar\lambda/10^6$ & $3759$ & $100$ & $60$ & $464$ & $15$ \\ \hline
\end{tabu}
\label{tab:configuration}
\end{table}

\vspace*{-2ex}

\begin{figure}[htbp]
    \centering      
    \setlength{\tabcolsep}{-3pt}
	\begin{tabular}{ccc}  
        \includegraphics[height=41mm, trim={2mm 8mm 2.8mm 4mm}, clip]{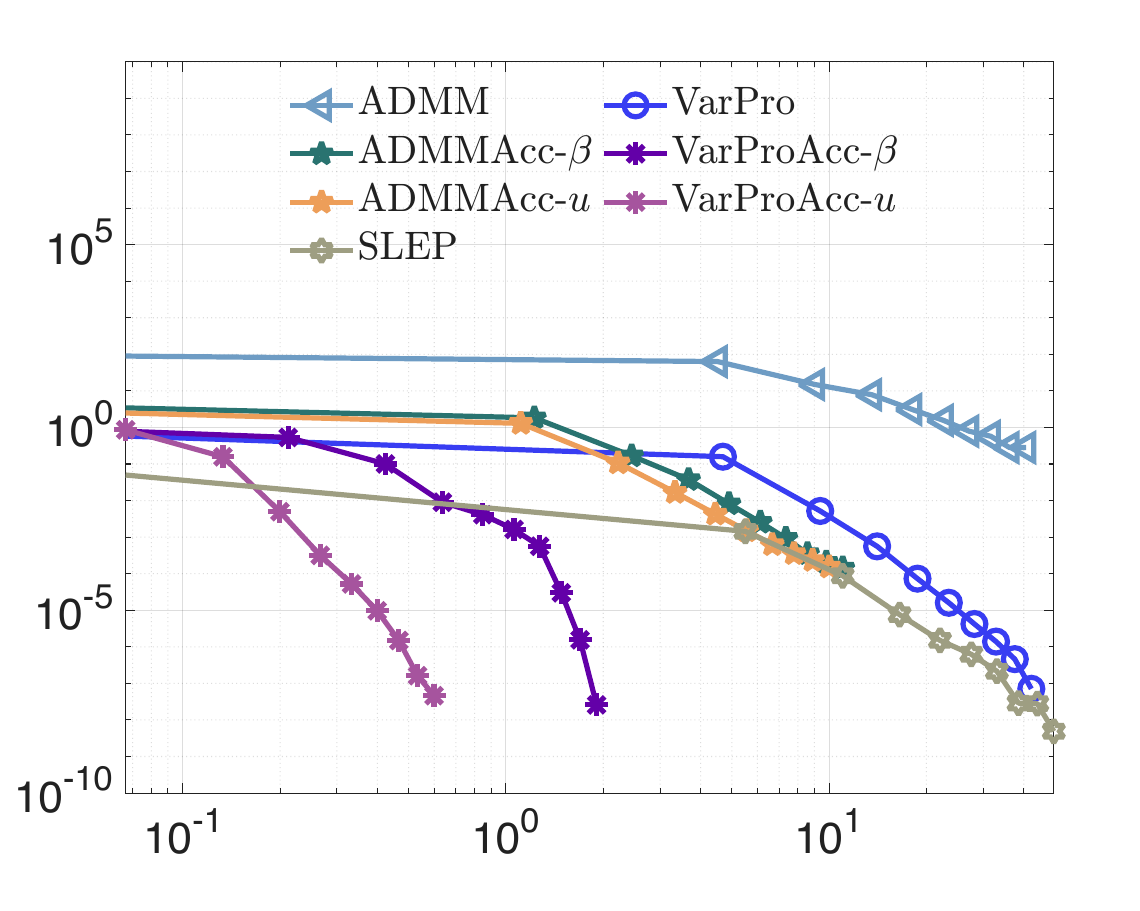} & 
        \includegraphics[height=41mm, trim={2mm 8mm 2.8mm 4mm}, clip]{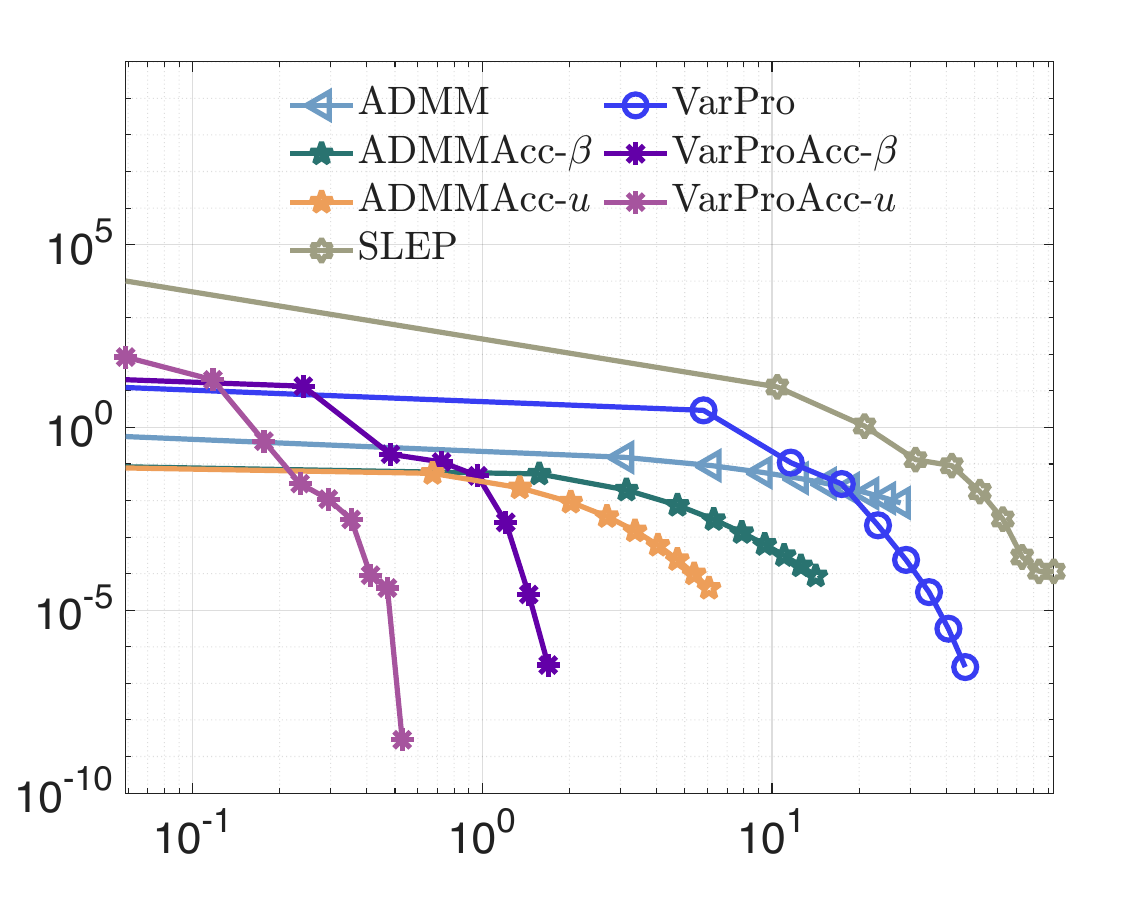}   &          
        \includegraphics[height=41mm, trim={2mm 8mm 2.8mm 4mm}, clip]{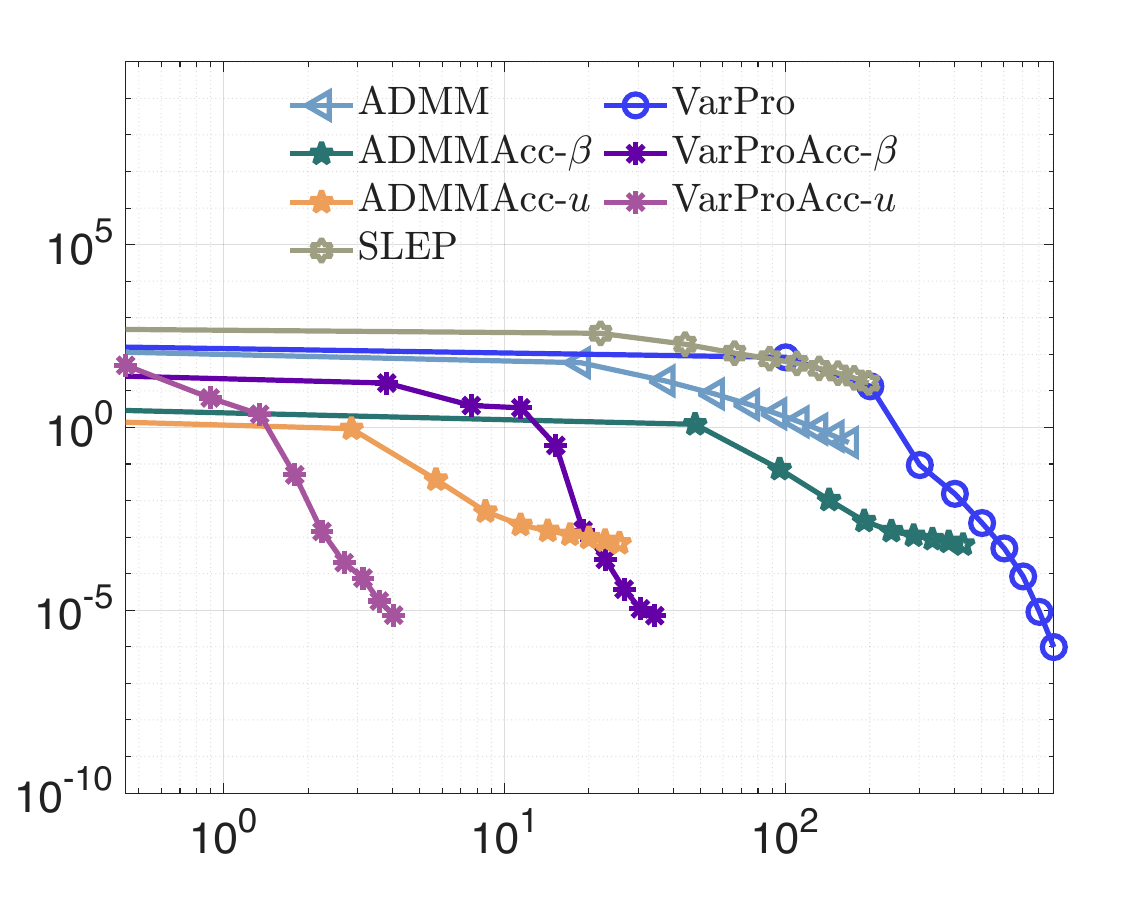} \\[-1mm]
        {\scriptsize (a). {\tt gisette} runtime} & {\scriptsize (b). {\tt E2006.test} runtime} & {\scriptsize (c). {\tt E2006.train} runtime}  \\[-0mm]
        \includegraphics[height=41mm, trim={2mm 8mm 2.8mm 4mm}, clip]{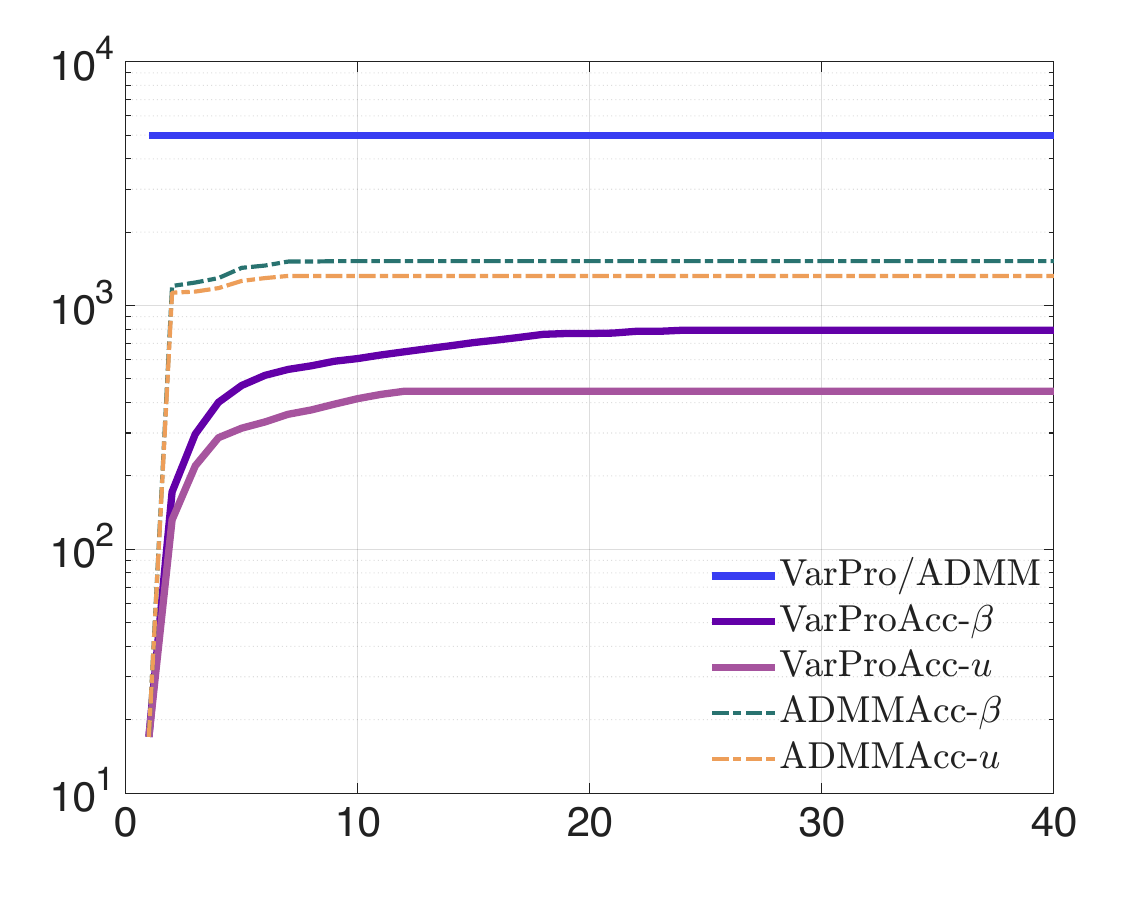}  & 
        \includegraphics[height=41mm, trim={2mm 8mm 2.8mm 4mm}, clip]{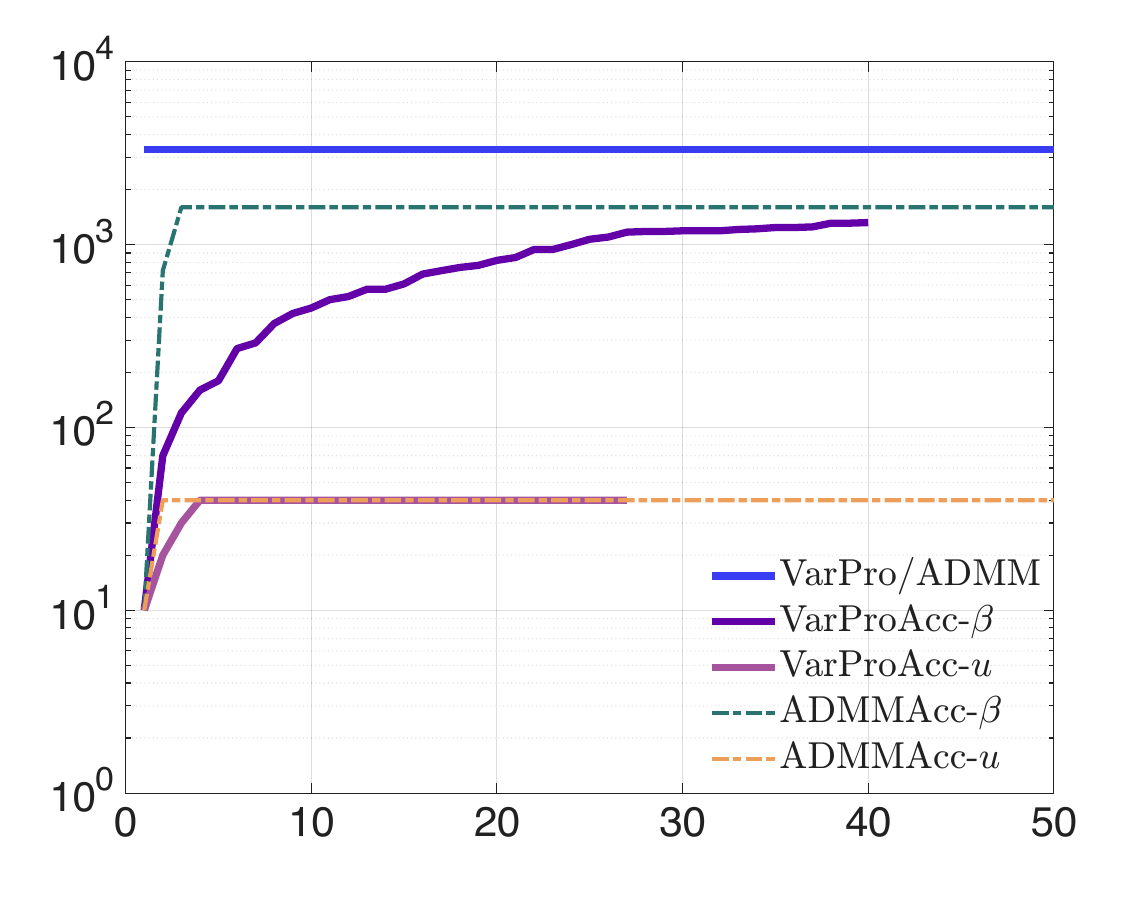}       &   
        \includegraphics[height=41mm, trim={2mm 8mm 2.8mm 4mm}, clip]{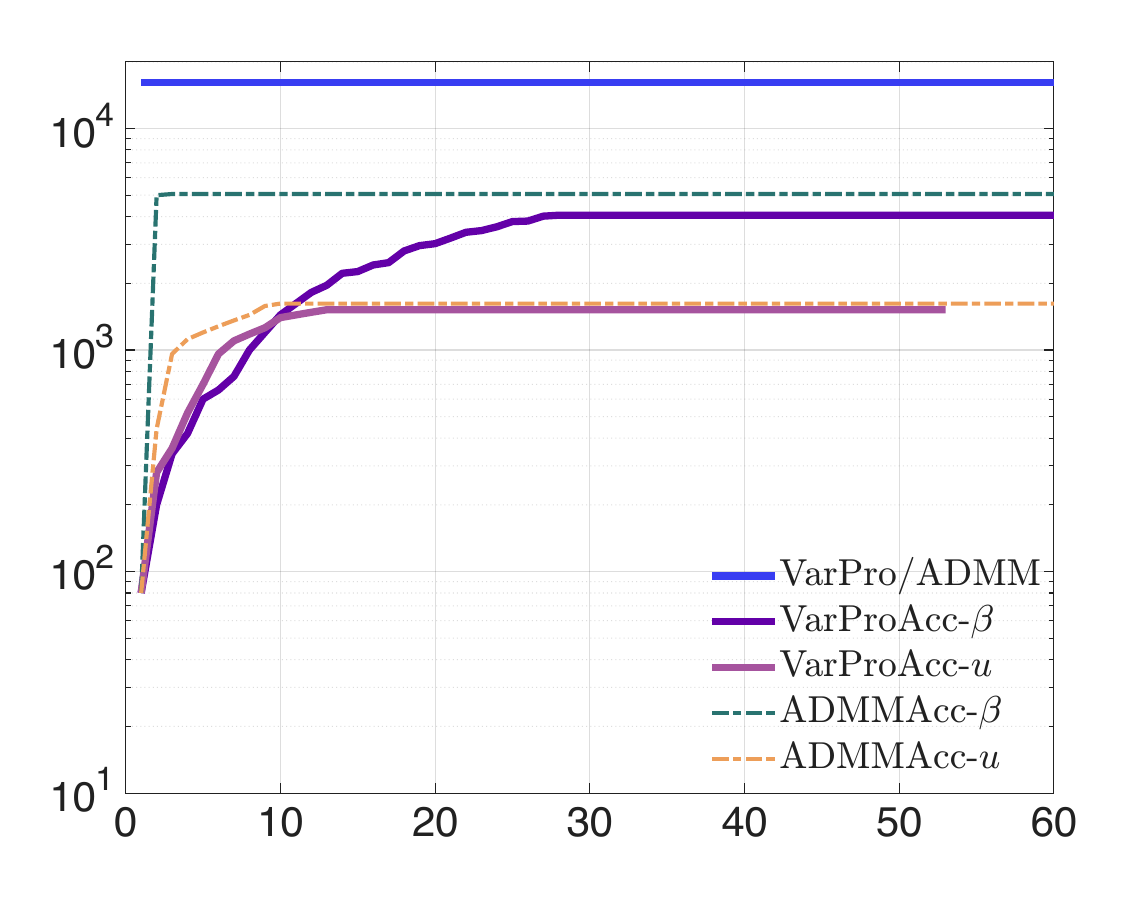}\\ [-1mm]
        {\scriptsize (a). {\tt gisette} $\abs{\Ee_{x}^{(k)}}$} & {\scriptsize (b). {\tt E2006.test} $\abs{\Ee_{x}^{(k)}}$} & {\scriptsize (c). {\tt E2006.train} $\abs{\Ee_{x}^{(k)}}$}
	\end{tabular}   \vspace*{-3pt}        
        \caption{Comparison of overlapping group LASSO over three datasets. First row: runtime; Second row: size of linear system.  
        }
        \label{fig:ogl:syn:1}
\end{figure}

The numerical comparison is provided in \cref{fig:ogl:syn:1} with first row comparing runtime and second row the dimension of linear systems of {\tt Option I/II}.
From the above comparison, we observe that
\begin{itemize}
    \item AdaDROPS significantly reduces the runtime of ADMM and VarPro. In particular, VarProAcc-$u$ provides more than an order speed-up. Coupled with AdaDROPS, both ADMM and VarPro also outperform the SLEP method.
    \item The {\tt Option II} which uses OGN certificate is better than the {\tt Option I} with LASSO certificate in detecting the nonzero groups, especially when the ratio of overlapping size against group size is large. 
\end{itemize}

\subsubsection{Wavelet-based image processing}

Let $M:\RR^{n\times n}\to\RR^p$ be a sparse random Gaussian matrix, given an image $u\in\RR^{n\times n}$, consider the observation
$
y = Mu + \varepsilon 
$  
where $\varepsilon$ stands for white Gaussian noise. To recover $u$ from $y$, one approach is the wavelet based sparse reconstruction model. Denote $W$ a wavelet transform and $W^{-1}$ the inverse transform, then the synthesis based approach \cite{elad2007analysis} to reconstruct $u$ reads
\begin{equation*} 
    \min_{x\in \RR^{n^2}} \sfrac{1}{2\lambda} \norm{MW^{-1} x - y}_2^2 + \norm{Lx}_{1,2},
\end{equation*}
where $x$ denotes the wavelet coefficients. 
The matrix $L$ defines the overlapping group structure, constructed according to the parent-child relationships of wavelet coefficients in a tree structure, as described in \cite{rao2011convex}.

\begin{figure}[htbp]
    \centering
    \setlength{\tabcolsep}{-4pt}
    \begin{tabular}{ccc}  
    \includegraphics[height=41mm, trim={2mm 8mm 2.8mm 4mm}, clip]{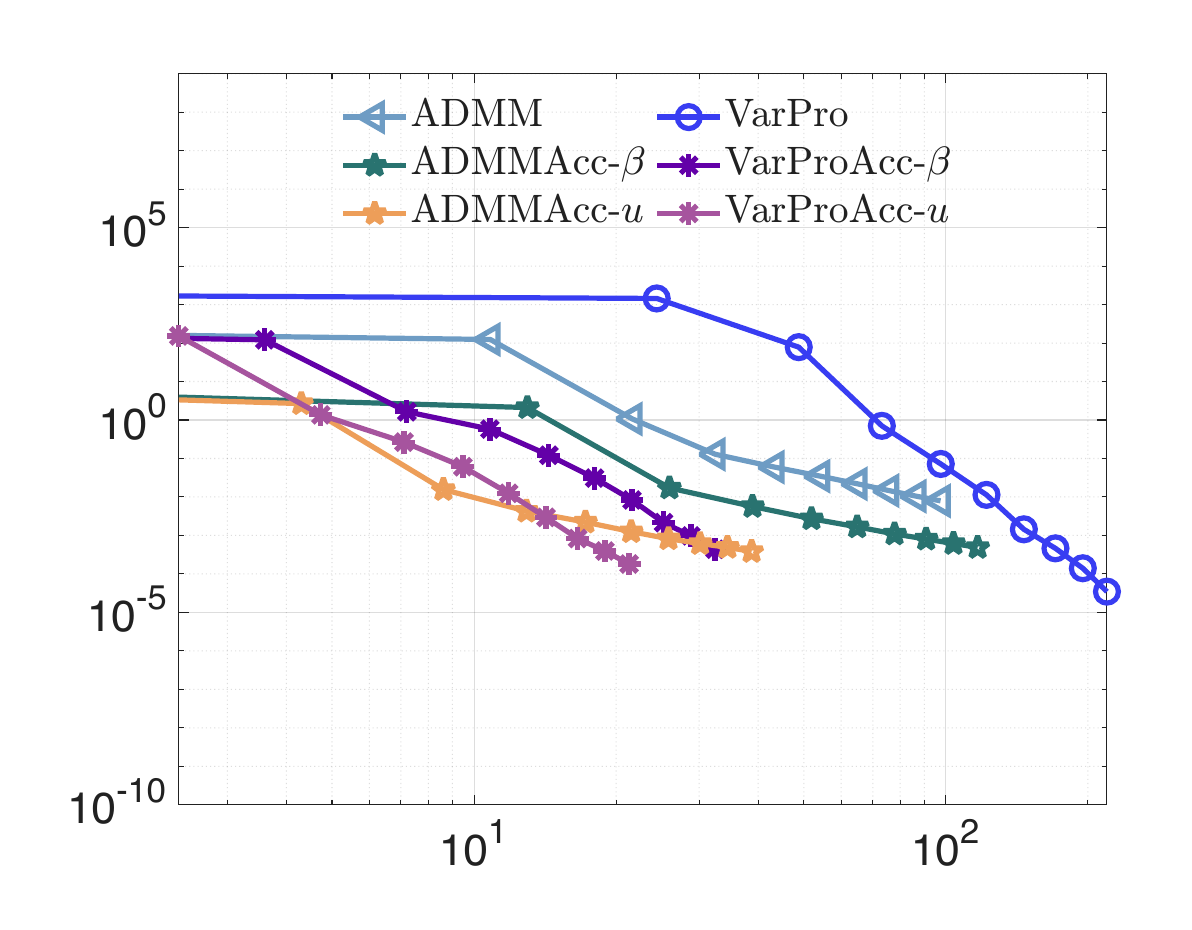} &
    \includegraphics[height=41mm, trim={2mm 8mm 2.8mm 4mm}, clip]{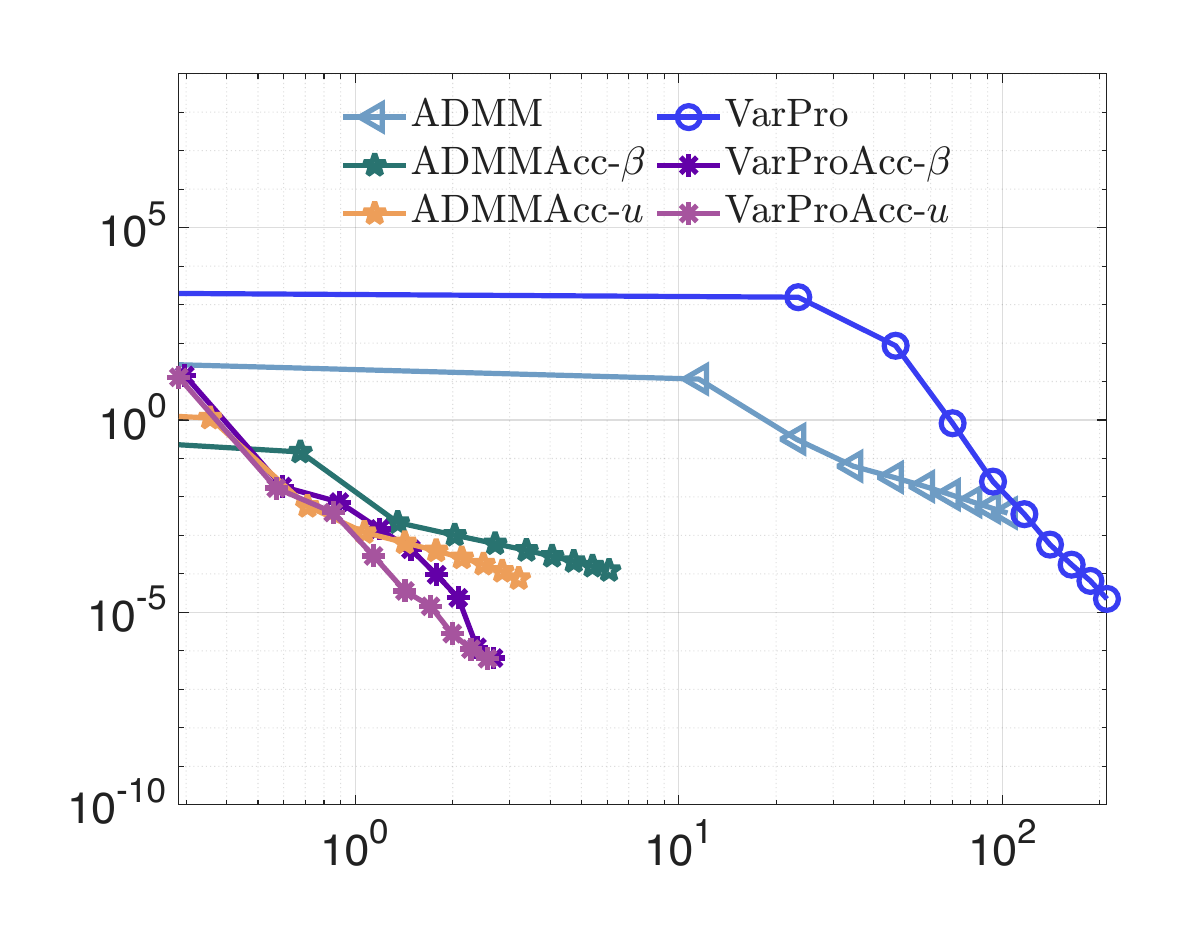} &
    \includegraphics[height=41mm, trim={2mm 8mm 2.8mm 4mm}, clip]{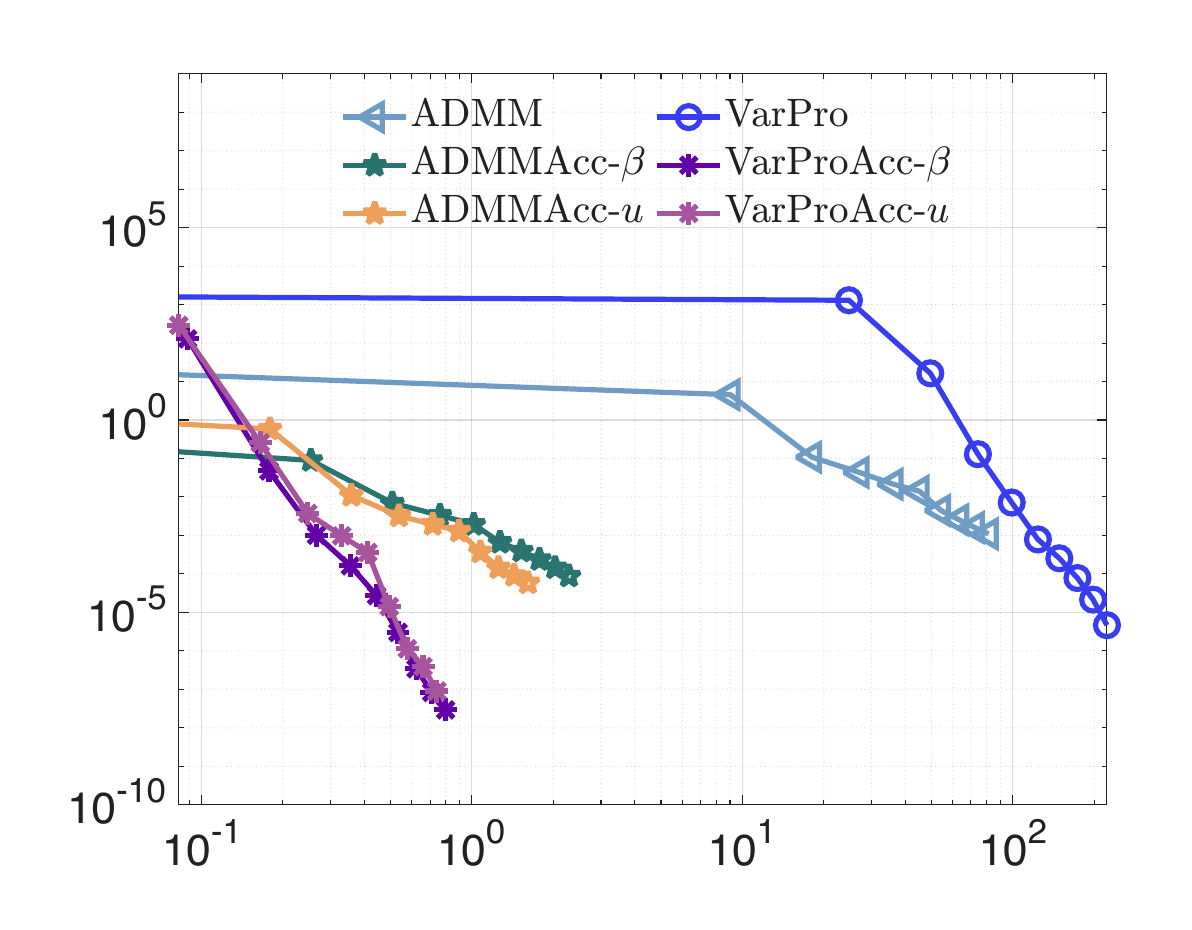} \\[-1mm]
    {\scriptsize (a). $n=128, \lambda=0.05$} & {\scriptsize (b). $n=128, \lambda=0.1$} & {\scriptsize (c). $n=128, \lambda=0.5$} \\[-0mm]
    \includegraphics[height=41mm, trim={2mm 8mm 2.8mm 4mm}, clip]{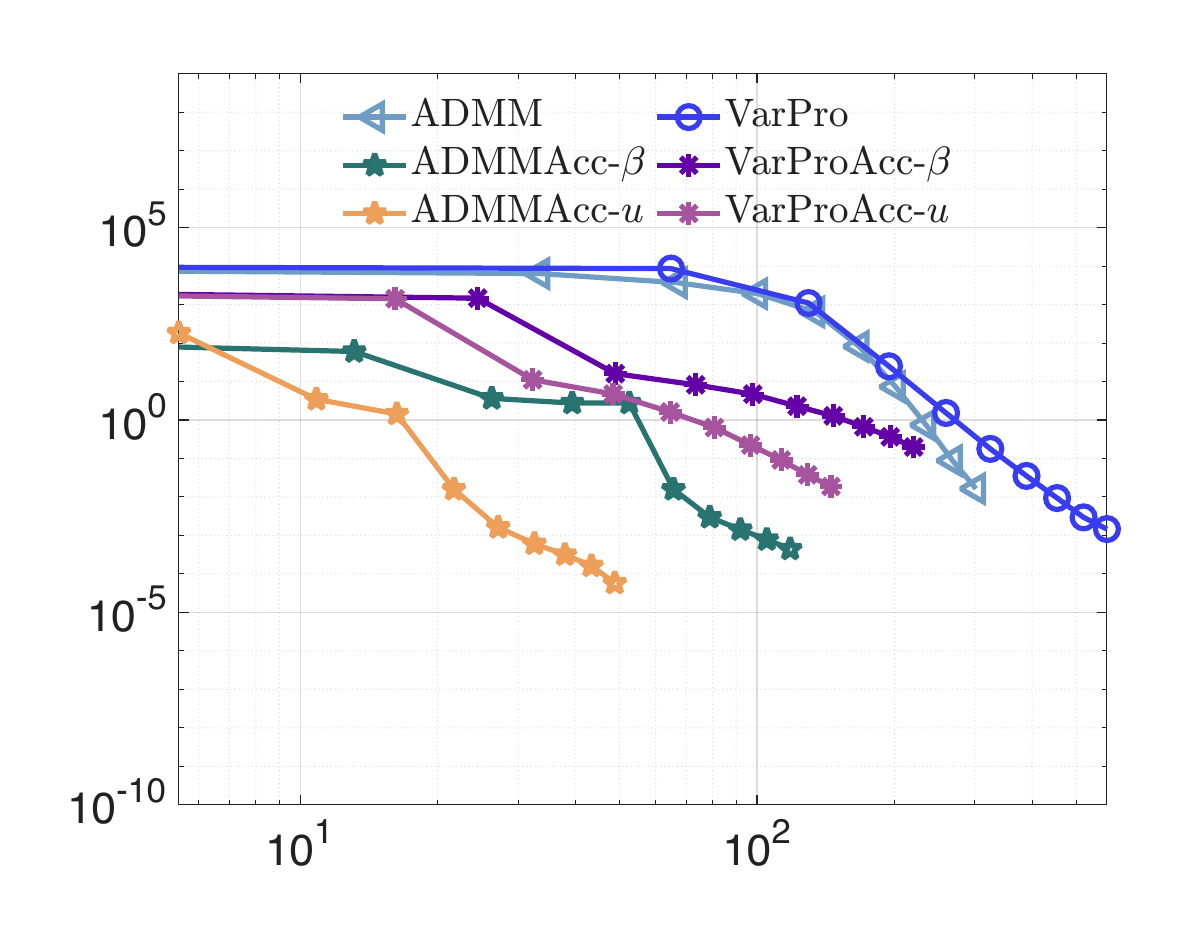} &  
    \includegraphics[height=41mm, trim={2mm 8mm 2.8mm 4mm}, clip]{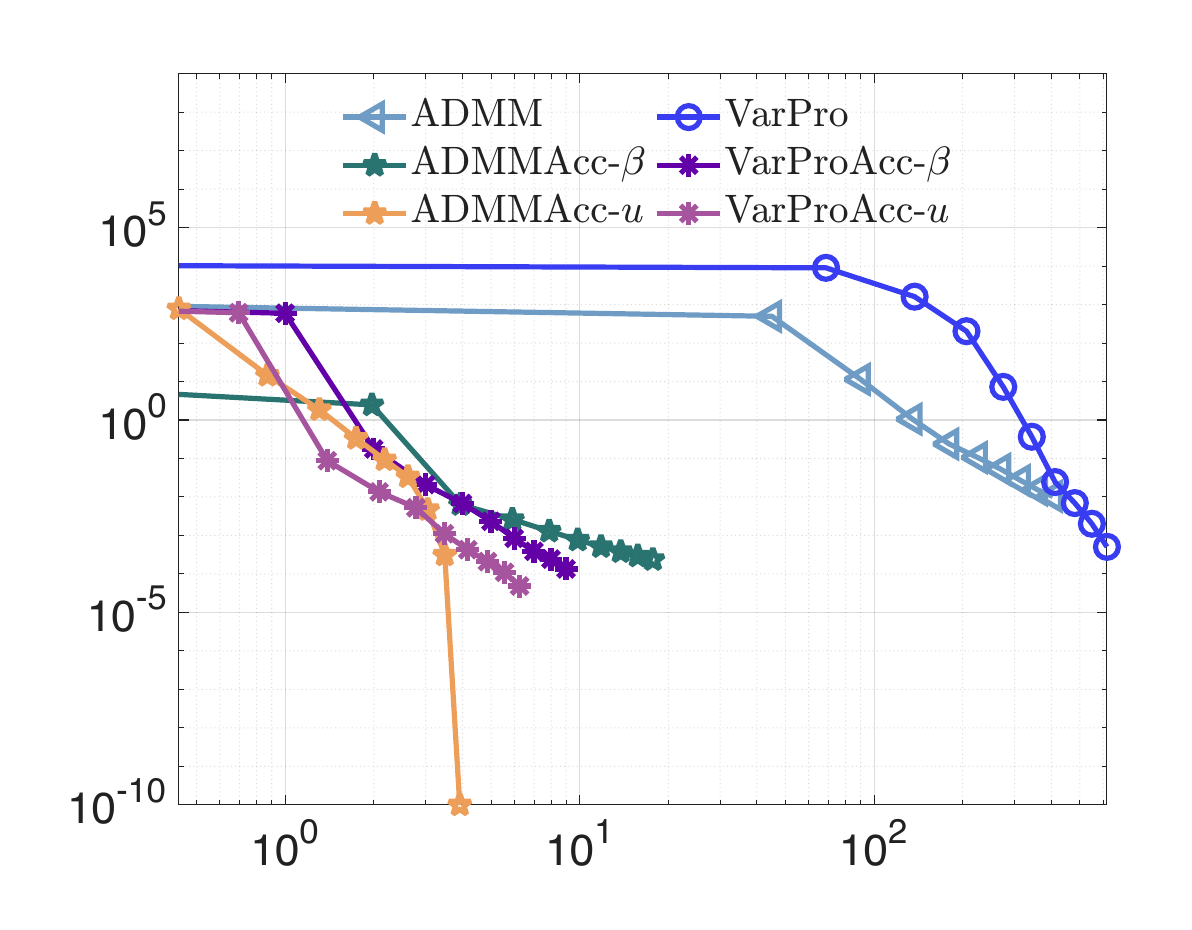}  & 
    \includegraphics[height=41mm, trim={2mm 8mm 2.8mm 4mm}, clip]{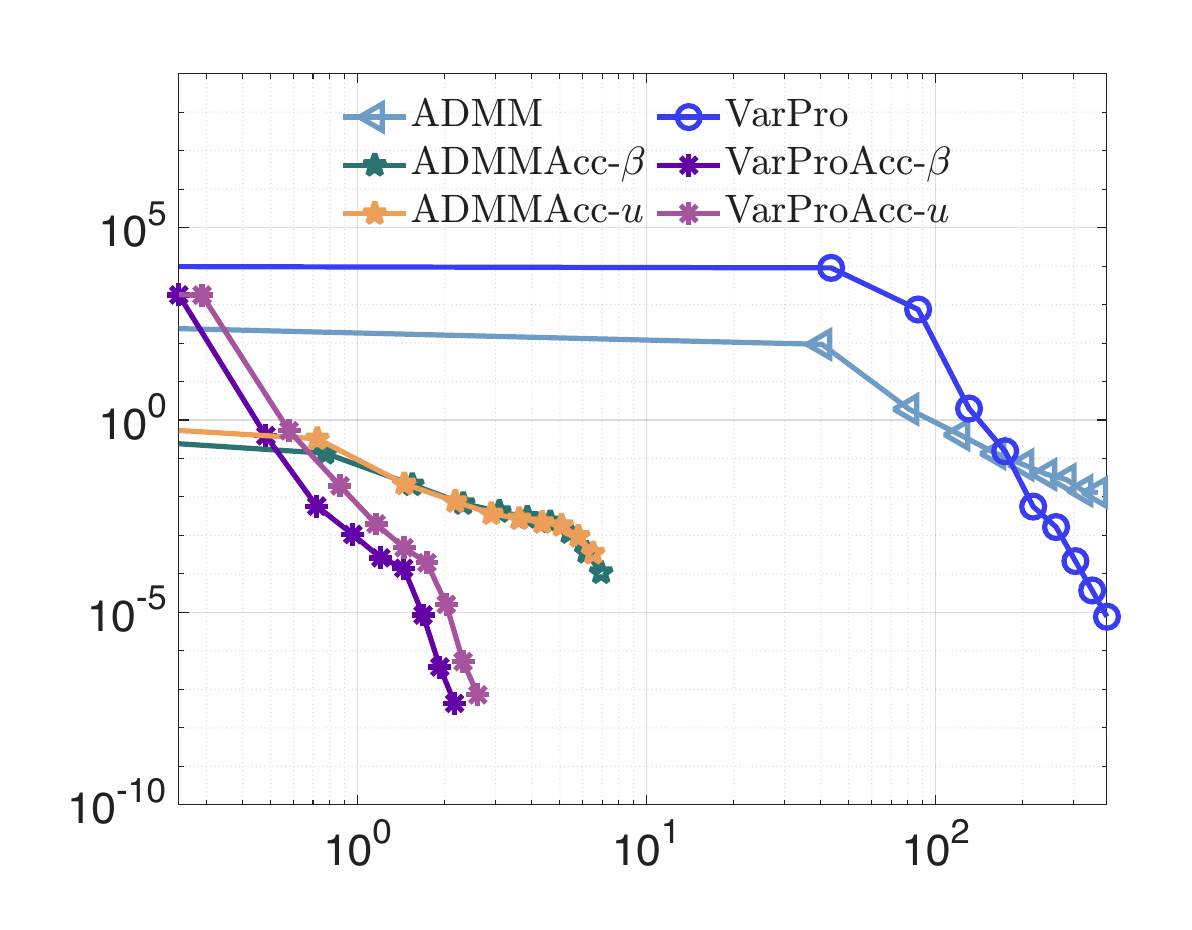} \\[-1mm]
    {\scriptsize (d). $n=256, \lambda=0.05$} & {\scriptsize (e). $n=256, \lambda=0.1$} & {\scriptsize (f). $n=256, \lambda=0.5$} \\
    \end{tabular}  \vspace*{-3pt} 
    \caption{Runtime comparison on two synthetic images. First row: $n=128$; Second row: $n=256$.}
    \label{fig:cs:recon}
    \vspace*{-2ex}
\end{figure}

In this experiment, the Haar wavelet is used. The settings of the problems are: i) Two synthetic images are considered with size $n = 128, 256$; ii) Matrix $M$ is generated by sparse random Gaussian distribution, for $n=128$ we set $p=\frac{128^2}{2}$, and for $n=256$ we choose $p=\frac{256^2}{4}$; iii) Wavelet decomposition level: $7$ for $n=128$ and $8$ for $n=256$.

Runtime comparison of ADMM/VarPro with/without AdaDROPS are provided in \cref{fig:cs:recon}, under different choices of regularization parameter $\lambda$
\begin{itemize}
    \item AdaDROPS consistently provides speedup over the standard algorithms.
    \item Between the two certificates, The OGN certificate outperforms the LASSO certificate for small $\lambda$; However, when $\lambda$ is large, e.g. 3rd column of \cref{fig:cs:recon}, both certificates provide comparable support estimation, the overhead of computing OGN certificate makes it slower than the LASSO certificate.
\end{itemize}

\subsection{Nonoverlapping sparsity}
We also consider nonoverlapping sparse optimization, e.g. $\ell_1$-norm and nonoverlapping group norm, on problems including LASSO, group LASSO and multi-task LASSO.

In the literature, dimension reduction for nonoverlapping sparsity are widely studied, three representative algorithms are considered for comparison, which are: semismooth Newton method (SSNAL)\footnote{\url{https://github.com/MatOpt/SuiteLasso}}\footnote{{\url{https://github.com/YangjingZhang/SparseGroupLasso}}} \cite{li2018highly,zhang2020efficient}, safe screening and dual extrapolation based  method CELER\footnote{\url{https://github.com/mathurinm/celer}} \cite{massias2018celer,massias2020dual} and the previously compared SLEP \cite{Liu:2009:SLEP:manual}. 
For AdaDROPS, for simplicity we consider only VarPro. As for nonoverlapping sparsity, the LASSO certificate and the OGN certificate are equivalent, hence we use ``VarProAcc'' to denote the method combined with AdaDROPS.

To ensure a consistent and fair comparison of algorithms with different stopping criteria, we adopt the relative KKT residual \cite{li2018highly,zhang2020efficient}, which is defined as
\[
    \eta = \sfrac{\| Lx - \mathrm{prox}_{\|\cdot \|_{1,2}} (Lx - \frac{1}{\lambda} L^{-1} A^\top (Ax-y))\|}{1 + \|Lx \|_{1,2} + \|Ax-y \|} .
\]
Note that for $\ell_1$-norm, we have $L=\Id$.

\subsubsection{LASSO}

We first consider LASSO problem, for which 8 datasets from LIBSVM as shown in \cref{table:lasso} are considered. From the comparison, we observe
\begin{itemize}
    \item Consistent with previous comparison, VarProAcc delivers significant acceleration over standard VarPro. In particular, for \texttt{log1p.E2006.train}, VarPro fails to converge in 20 minutes, while VarProAcc only needs 20 seconds. 
    
    \item For the compared methods, CELER overall is the fastest. The key of CELER is its ability to identify the sparsity of the solution. 
\end{itemize}

\begin{table}[htbp]
\caption{Time and accuracy comparison on different datasets for LASSO problem. 
For each dataset, $\lambda=\lambda_{\max}/r$ with $\lambda_{\max} = \norm{A^\top y}_{\infty}$ and $r$ specified below.
The time unit is ``second'', and ``-'' means algorithm failing to converge in 1200 seconds. The {\color{cell_red}{red}} / {\color{cell_blue}{blue}} markers represent the fastest/second fastest in terms of the running time.}\vspace*{-1ex}
\begin{adjustbox}{width=\textwidth,totalheight=\textheight,keepaspectratio}
\begin{small}
\tabulinesep=.5mm
\begin{tabu}{|cl|c|c|c|c|c|c|c|c|}
\hline
\multicolumn{2}{|c|}{data $(m,n)$} & $r$ & $\abs{\supp(\xsol)}$  &   & VarPro  & VarProAcc  & SSNAL  & SLEP   & CELER  \\ \hline
\multicolumn{2}{|c|}{\multirow{2}{*}{\begin{tabular}[c]{@{}c@{}}E2006.test\\ (3308,150358)\end{tabular}}}  & \multirow{2}{*}{$10^3$}   & \multirow{2}{*}{1}  & $\eta$ & 7.7e-8  & 4.7e-8  & 3.4e-7 & 7.5e-12 & 1.2e-12 \\ \cline{5-10} 
\multicolumn{2}{|c|}{}   &  &   & time   & 11.78  & 0.05  & 0.21   & {\color{cell_red}0.01}   & {\color{cell_blue}0.03}    \\ \hline
\multicolumn{2}{|c|}{\multirow{2}{*}{\begin{tabular}[c]{@{}c@{}}E2006.train\\ (16087,150360)\end{tabular}}} &\multirow{2}{*}{$10^3$}   & \multirow{2}{*}{1}  & $\eta$ & 1.9e-9 & 8.2e-10 & 6.4e-7 & 5.8e-5 & 1.8e-11 \\ \cline{5-10} 
\multicolumn{2}{|c|}{}     &    &    & time   & 242.54  & {\color{cell_red}0.07}    & 0.48   & 0.71   & {\color{cell_blue}0.09}    \\ \hline
\multicolumn{2}{|c|}{\multirow{2}{*}{\begin{tabular}[c]{@{}c@{}}log1p.E2006.test\\ (3308,4272226)\end{tabular}}} &\multirow{2}{*}{$10^3$}   & \multirow{2}{*}{8}  & $\eta$ & 2.5e-4  & 6.7e-4  & 1.3e-5 & 5e-2   & 7e-4    \\ \cline{5-10} 
\multicolumn{2}{|c|}{}    &    &    & time   & 198.24   & 13.14   & {\color{cell_blue}3.98}   & 449.00 & {\color{cell_red}3.00}    \\ \hline
\multicolumn{2}{|c|}{\multirow{2}{*}{\begin{tabular}[c]{@{}c@{}}log1p.E2006.train\\ (16087,4272227)\end{tabular}}} &\multirow{2}{*}{$10^3$} & \multirow{2}{*}{5}  & $\eta$ & -     & 4.6e-5    & 3.8e-5 & - & 8.6e-4  \\ \cline{5-10} 
\multicolumn{2}{|c|}{}   &  &  & time   & -    & 20.52   & {\color{cell_blue}7.32}   & -   & {\color{cell_red}5.14}    \\ \hline
\multicolumn{2}{|c|}{\multirow{2}{*}{\begin{tabular}[c]{@{}c@{}}abalone7\\ (4177,6435)\end{tabular}}} &\multirow{2}{*}{$10^3$}  & \multirow{2}{*}{24} & $\eta$ & 4.3e-5  & 8.3e-5  & 6.5e-6 & 3.0e-3 & 2.3e-5  \\ \cline{5-10} 
\multicolumn{2}{|c|}{}  &   &  & time   & 37.12   & {\color{cell_red}0.56}    & {\color{cell_blue}1.26}   & 101.73  & 3.08    \\ \hline
\multicolumn{2}{|c|}{\multirow{2}{*}{\begin{tabular}[c]{@{}c@{}}space\_ga9\\ (3107,5005)\end{tabular}}} &\multirow{2}{*}{$10^3$}   & \multirow{2}{*}{14} & $\eta$ & 7.9e-6  & 2.0e-5  & 8.0e-7 & 2.3e-4 & 4.4e-5  \\ \cline{5-10} 
\multicolumn{2}{|c|}{}   &   &  & time   & 14.00  & {\color{cell_red}0.05}    & 0.52   & 26.07  & {\color{cell_blue}0.20}    \\ \hline
\multicolumn{2}{|c|}{\multirow{2}{*}{\begin{tabular}[c]{@{}c@{}}bodyfat7\\ (252,116280)\end{tabular}}}   &\multirow{2}{*}{$10^3$}           & \multirow{2}{*}{2}  & $\eta$ & 1.1e-6  & 2.0e-6  & 4.9e-8 & 5e-3 & 9.1e-13 \\ \cline{5-10} 
\multicolumn{2}{|c|}{}  &  &   & time   & 1.81    & {\color{cell_red}0.29}   & 0.85   & 72.20  & {\color{cell_blue}0.33}    \\ \hline
\multicolumn{2}{|c|}{\multirow{2}{*}{\begin{tabular}[c]{@{}c@{}}rcv1\\ (20242,47236)\end{tabular}}}      & \multirow{2}{*}{2}           & \multirow{2}{*}{3}  & $\eta$ & 2.2e-6   & 2.5e-6  & 8.7e-7 & 2.5e-7 & 6.3e-13 \\ \cline{5-10} 
\multicolumn{2}{|c|}{}   &     &  & time   & 92.06   & {\color{cell_blue}0.06}    & 0.12   & 0.22   & {\color{cell_red}0.02}    \\ \hline
\end{tabu}
\label{table:lasso}
\end{small}
\end{adjustbox}
\end{table}

\subsubsection{Group LASSO}

We continue using the datasets from LASSO example to compare the above methods on nonoverlapping group LASSO, and the result is provided in \cref{table:grpLasso}:
\begin{itemize}
    \item VarProAcc again outperforms the standard VarPro. 
    \item CELER performs quite poorly on the large-scale datasets, for which SLEP also struggles. 
    \item SSNAL is the best of the three compared methods, especially for the two large-scale datasets.
\end{itemize} 
Overall, VarProAcc provides robust and comparable performance.

\begin{table}[!htb]
\caption{Time and accuracy comparison on different datasets for group LASSO problem. For each dataset, $\lambda=\lambda_{\max}/r$ with $\lambda_{\max} = \max_{i \in \dbrack{\Nn}}\norm{A^\top_{G_i} y}/w_i$ and $r$ specified below. {\tt nnz(g)} denotes the number of nonzero entries and groups of $\xsol$. The time unit is ``second'', and ``-'' means algorithm failing to converge in 1200 seconds. The {\color{cell_red}{red}} / {\color{cell_blue}{blue}} markers represent the fastest/second fastest in terms of the running time.}\vspace*{-1ex}
\begin{adjustbox}{width=\textwidth,totalheight=\textheight,keepaspectratio}
\tabulinesep=.5mm
\begin{tabular}{|c|c|c|c|c|c|c|c|c|}
\hline
data $(m,n,\mathcal{N})$  & $r$ & {\tt nnz(g)}  &  & VarPro  & VarProAcc  & SSNAL & SLEP  & CELER \\ \hline
\multirow{2}{*}{\begin{tabular}[c]{@{}c@{}}E2006.test\\ (3308,150358,5012)\end{tabular}}   & \multirow{2}{*}{$10^3$} & \multirow{2}{*}{5(1)}  & $\eta$ & 2.5e-10 & 2.5e-10  & 2.0e-8 & 3.5e-6  & 6.6e-5 \\ \cline{4-9}          &                         &                          & time   & 19.93     & {\color{cell_red}0.22}    & {\color{cell_blue}0.32}    & 1.51    & 19.21  \\ \hline
\multirow{2}{*}{\begin{tabular}[c]{@{}c@{}}E2006.train\\ (16087,150360,5012)\end{tabular}}          & \multirow{2}{*}{$10^3$} & \multirow{2}{*}{4(1)}    & $\eta$ & 5.0e-9  & 6.3e-12 & 6.1e-10 & 7.7e-9  & 7.6e-5 \\ \cline{4-9}          &                         &                          & time   & 382.40    & {\color{cell_red}0.60}    & {\color{cell_blue}0.64}    & 16.62    & 20.18  \\ \hline
\multirow{2}{*}{\begin{tabular}[c]{@{}c@{}}log1p.E2006.test\\ (3308,4272226,142408)\end{tabular}}   & \multirow{2}{*}{$10^3$} & \multirow{2}{*}{263(1)}  & $\eta$ & 2.5e-10        & 1.6e-8  & 1.6e-8 & 3.1e-5  & -      \\ \cline{4-9}          &                         &                          & time   & 18.75       & {\color{cell_blue}4.09}    & {\color{cell_red}2.95}    & 682.68  & -      \\ \hline
\multirow{2}{*}{\begin{tabular}[c]{@{}c@{}}log1p.E2006.train\\ (16087,4272227,142408)\end{tabular}} & \multirow{2}{*}{$10^3$} & \multirow{2}{*}{339(1)}        & $\eta$ & -      & 4.1e-9  & 4.4e-9   & -       & -      \\ \cline{4-9}         &                         &                          & time   & -       & {\color{cell_blue}7.80}    & {\color{cell_red}5.88}   & -       & -      \\ \hline
\multirow{2}{*}{\begin{tabular}[c]{@{}c@{}}gisette\\ (6000,5000,167)\end{tabular}}                  & \multirow{2}{*}{$5$}    & \multirow{2}{*}{322(13)} & $\eta$ & 1.0e-8   & 5.0e-9  & 2.3e-9  & 1.4e-7  & 6.5e-8 \\ \cline{4-9}         
&                         &                          & time   & 29.99      & {1.60}    & {\color{cell_blue}1.59}    & 20.17   & {\color{cell_red}1.03}   \\ \hline
\multirow{2}{*}{\begin{tabular}[c]{@{}c@{}}abalone7\\ (4177,6435,215)\end{tabular}}                 & \multirow{2}{*}{$5$}    & \multirow{2}{*}{62(2)}   & $\eta$ & 1.7e-9   & 9.4e-10 & 1.6e-9  & 1.1e-8 & 2.9e-7 \\ \cline{4-9}         &                         &                          & time   & 13.96    & {\color{cell_red}0.25}    & {\color{cell_blue}0.57}    & 29.91   & 0.78   \\ \hline
\multirow{2}{*}{\begin{tabular}[c]{@{}c@{}}space\_ga9\\ (3107,5005,250)\end{tabular}}               & \multirow{2}{*}{$5$}    & \multirow{2}{*}{18(1)}   & $\eta$ & 6.4e-10  & 2.3e-9  & 4.6e-9  & 1.5e-7  & 3.6e-5 \\ \cline{4-9}          &                         &                          & time   & 5.23      & {\color{cell_red}0.10}    & {\color{cell_blue}0.24}    & 0.37    & 0.32   \\ \hline
\multirow{2}{*}{\begin{tabular}[c]{@{}c@{}}bodyfat7\\ (252,116280,1163)\end{tabular}}               & \multirow{2}{*}{$5$}    & \multirow{2}{*}{86(1)}   & $\eta$ & 6.1e-8   & 4.8e-9  & 2.0e-7  & 7.4e-6  & 1.1e-4 \\ \cline{4-9}          &                         &                          & time   & 1.03     & {\color{cell_red}0.18}    & {\color{cell_blue}0.50}    & 14.70   & 1.97   \\ \hline
\end{tabular}
\label{table:grpLasso}
\end{adjustbox}
\end{table}

\subsubsection{Multi-task LASSO} 
In multi-task LASSO \cite{gramfort2012mixed}, one aims to solve the following problem
\[
\min_{X\in\RR^{n\times q}} \sfrac{1}{2\lambda}\norm{AX-Y}_F^2 + \msum_{j=1}^{n} \norm{x^j} , 
\]
where $\norm{\cdot}_F$ stands for Frobenius norm, $A\in\RR^{m\times n}$, $Y\in\RR^{m\times q}$ and $x^j,j=1,...,n$ stands for the rows of $X$.

Two datasets are considered for experiments, one synthetic random Gaussian dataset and a joint magnetoencephalography (MEG)/electroencephalography (EEG) dataset from \cite{ndiaye2015gap}, their specific settings are in the table below. The number of nonzero groups of two datasets under different choices of $\lambda$ is also provided. 

\begin{table}[!htb]
\label{tab:configuration2}
\centering
\small
\caption{Configurations of the two datasets. Denote $\lambda_{\max} = \max_j \norm{x^j Y}_2$. 
}\vspace*{-1ex}
\tabulinesep=.75mm
\begin{tabu}{c|ccc|ccc}
\hline
                    & \multicolumn{3}{c|}{${\tt synthetic}$}                                                                  & \multicolumn{3}{c}{${\tt MEG/EEG}$}                                                                     \\ \hline
$(m,n,q)$           & \multicolumn{3}{c|}{$(1000,10000,80)$}                                                                  & \multicolumn{3}{c}{$(305,22494,85)$}                                                                    \\ \hline
$\lambda$           & \multicolumn{1}{c|}{$\lambda_{\max}/5$} & \multicolumn{1}{c|}{$\lambda_{\max}/10$} & $\lambda_{\max}/15$ & \multicolumn{1}{c|}{$\lambda_{\max}/2$} & \multicolumn{1}{c|}{$\lambda_{\max}/5$} & $\lambda_{\max}/10$ \\ \hline
$\abs{\Ii_{\xsol}}$ & \multicolumn{1}{c|}{$30$}               & \multicolumn{1}{c|}{$60$}               & $256$                & \multicolumn{1}{c|}{$5$}                & \multicolumn{1}{c|}{$20$}               & $76$                \\ \hline
\end{tabu}
\end{table}

\cref{MTL : syn} demonstrates the performance of VarPro w./w.o. AdaDROPS compared with SLEP and CELER:
\begin{itemize}
    \item For both datasets, VarProAcc is faster than the standard VarPro. Especially when $\lambda$ is large. Note that for MEG/EEG dataset, the advantage of VarProAcc is limited, which is due to the reason that $m$ is not large. 
    \item The performances of SLEP and CELER are mixed, with CELER faster on MEG/EEG data and SLEP faster on the synthetic data. Both are slower than VarProAcc. 
\end{itemize}

\begin{figure}[htbp]
    \centering
    \setlength{\tabcolsep}{-2.5pt}
    \begin{tabular}{ccc}
        \includegraphics[height=41mm, trim={2mm 8mm 2.8mm 4mm}, clip]{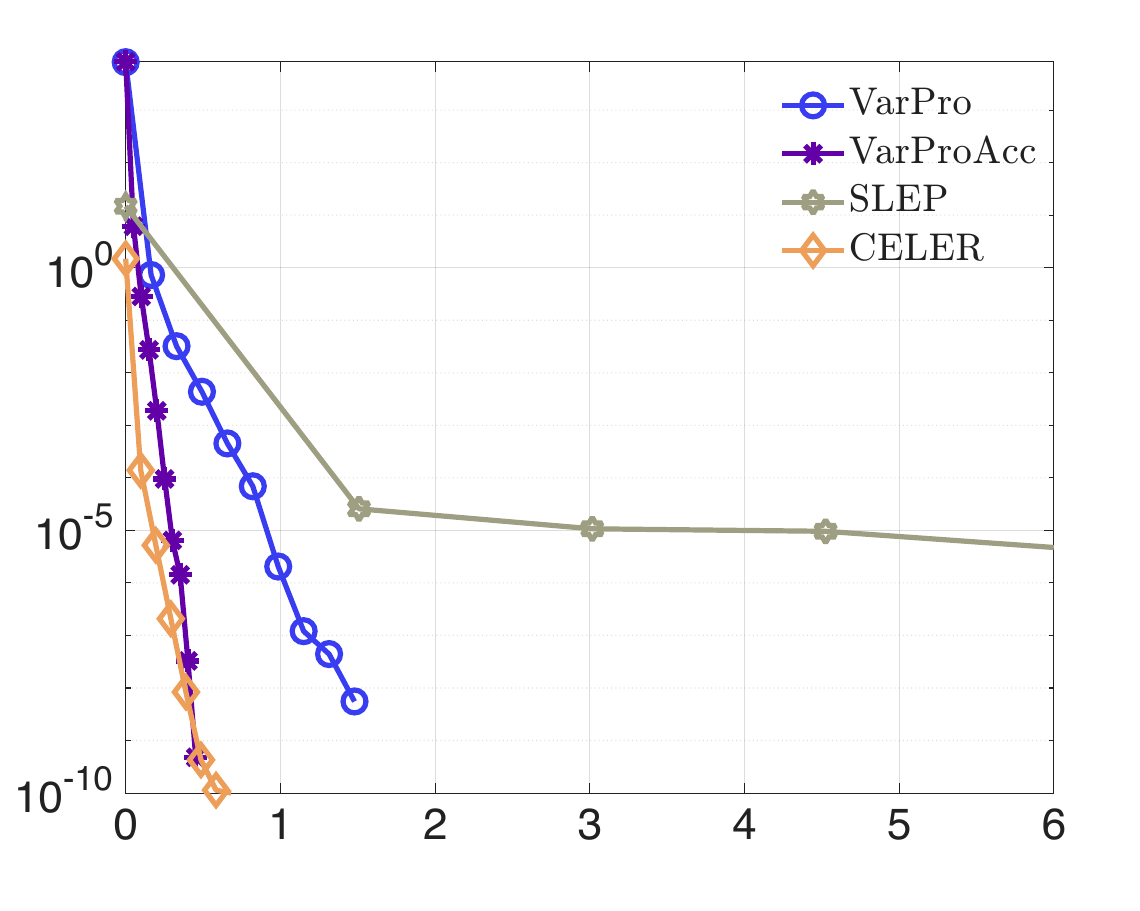}&
        \includegraphics[height=41mm, trim={2mm 8mm 2.8mm 4mm}, clip]{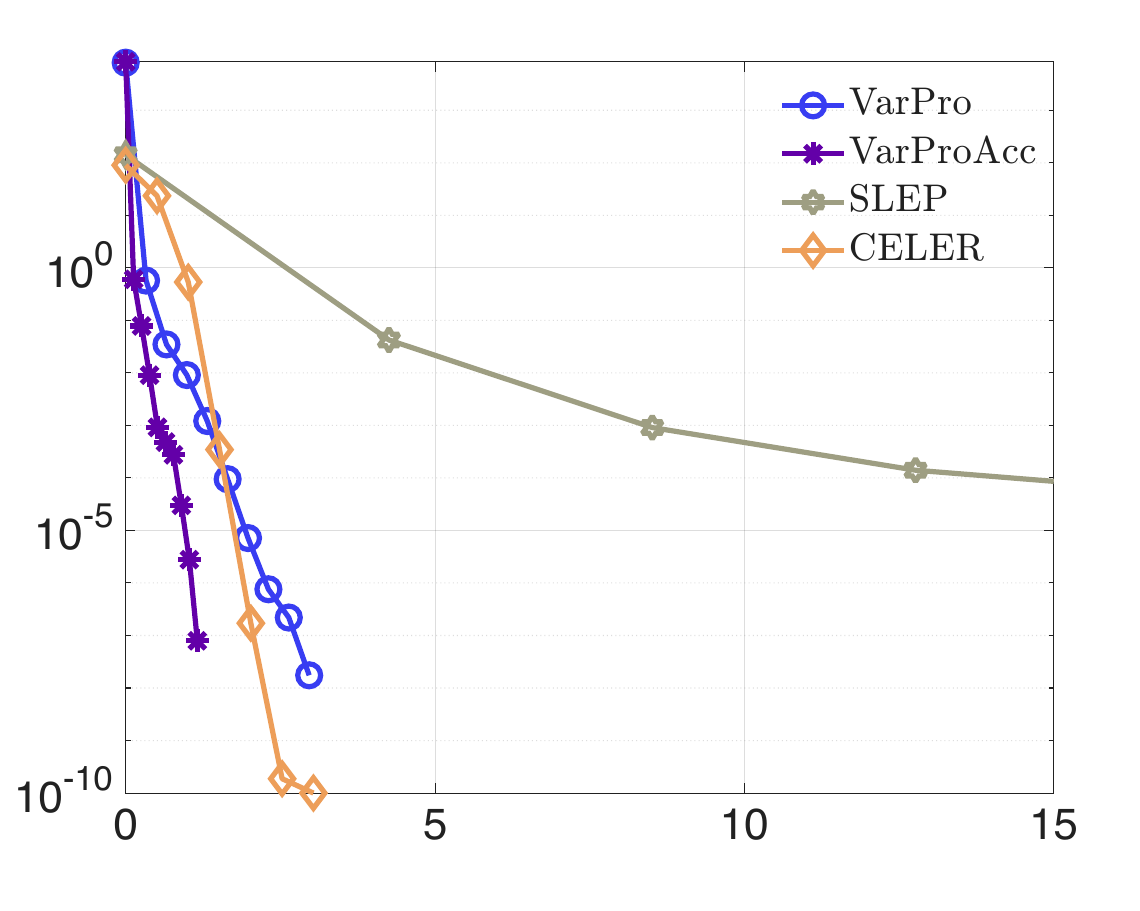}&
        \includegraphics[height=41mm, trim={2mm 8mm 2.8mm 4mm}, clip]{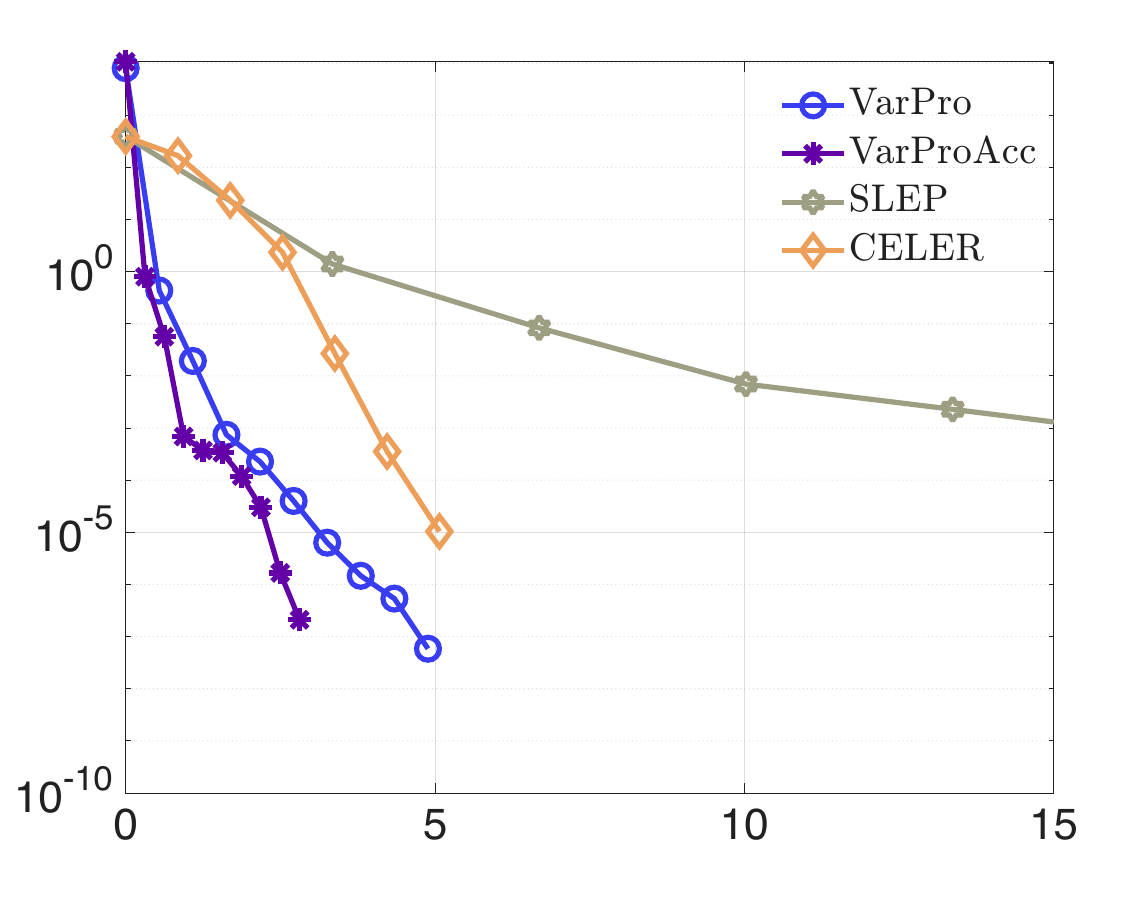}\\[-1mm]
        {\scriptsize (d). {\tt MEG/EEG}, $\lambda_{\max}/2$} & {\scriptsize (e). {\tt MEG/EEG}, $\lambda_{\max}/5$} & {\scriptsize (f). {\tt MEG/EEG}, $\lambda_{\max}/10$} \\
        \includegraphics[height=41mm, trim={2mm 8mm 2.8mm 4mm}, clip]{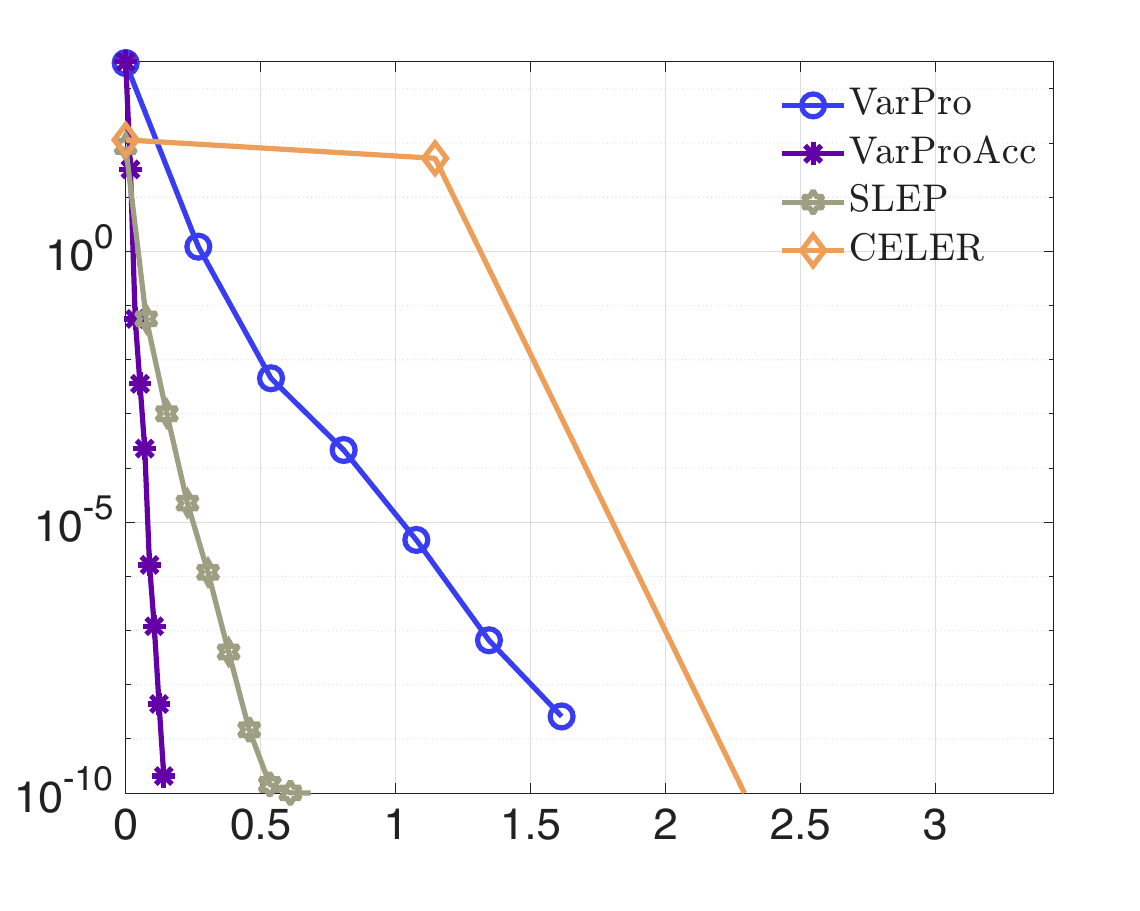}&
        \includegraphics[height=41mm, trim={2mm 8mm 2.8mm 4mm}, clip]{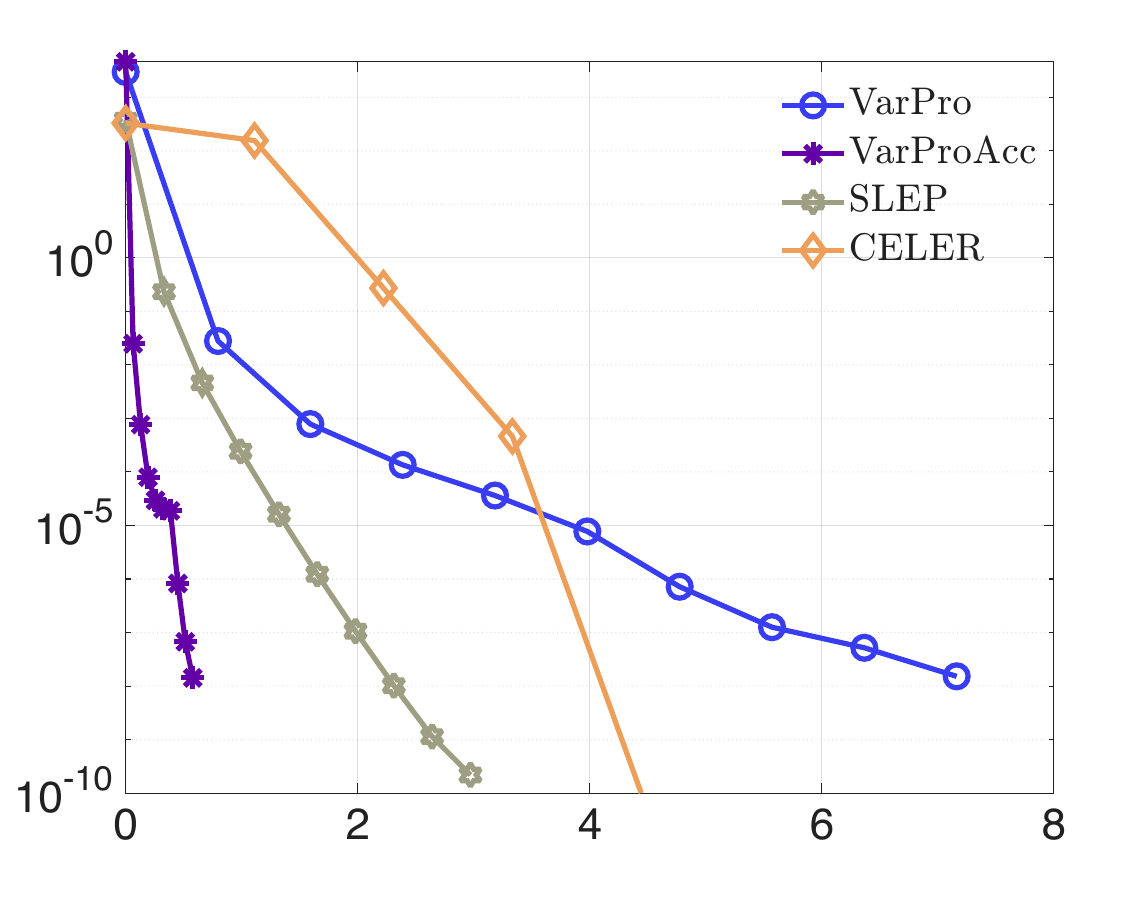}&
        \includegraphics[height=41mm, trim={2mm 8mm 2.8mm 4mm}, clip]{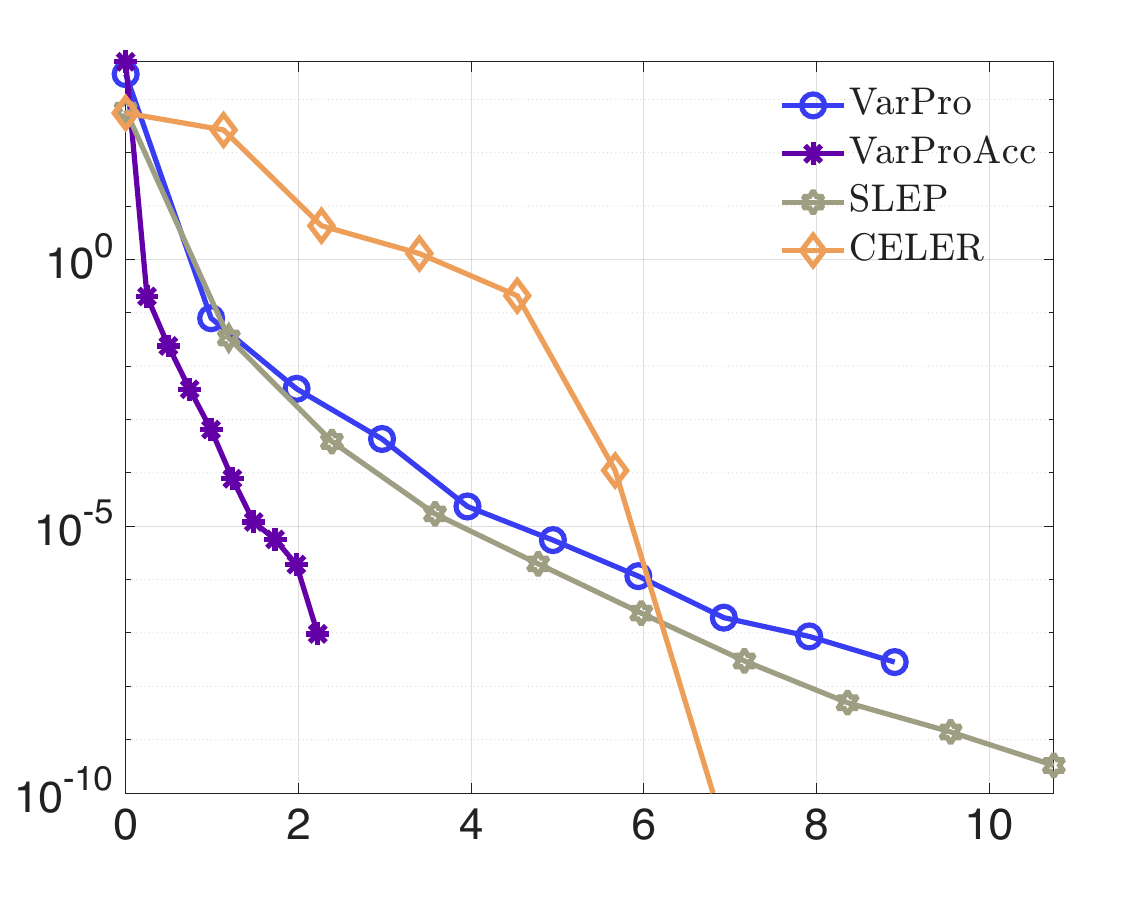} \\[-1mm]
        {\scriptsize (a). {\tt synthetic}, $\lambda_{\max}/5$} & {\scriptsize (b). {\tt synthetic}, $\lambda_{\max}/10$} & {\scriptsize (c). {\tt synthetic}, $\lambda_{\max}/15$} 
    \end{tabular}
    \caption{Comparison of multi-task LASSO on different datasets.}
    \label{MTL : syn}
    \vspace*{-2ex}
\end{figure}

\section{Conclusion}
\label{sec:conclusion}

In this paper, we propose dual certificates for overlapping group sparsity, namely LASSO certificate and OGN certificate, with the latter providing more efficient support identification performance. Based on these certificates, we introduce AdaDROPS, an adaptive dimension reduction scheme that can be combined with existing solvers. Integrating AdaDROPS with ADMM and VarPro significantly improves computational efficiency on popular datasets, in both overlapping and nonoverlapping settings. 
A future direction is to generalize the adaptive dimension reduction framework to more complicated problems, such as total variation, nuclear norm, and even nonsmooth model like TV-L1.

\appendix
\crefalias{section}{appendix}
\crefalias{subsection}{appendix}

\section{Properties of the lifting operator}
\label{sec:ogn_properties}

Here we collect some necessary properties related to $L$, projection operators and effective lifting operator $\widehat{L}$ defined in \cref{sec:overlapping_group_norm}.

\subsection{Facts of lifting operator \texorpdfstring{$L$}{L}}
\label{sec:fact:L}
\begin{fact}
\label{fact:L} The group operator $L$ has full column rank, and 
\begin{enumerate}[label={\rm \alph*).}]
    \item Each row of $L$ has only $1$ nonzero entry;
    \item For $i\in\dbrack{n}$, the number of nonzero entries of the $i$'th column of $L$ represents the number of groups that the $i$-th entry belongs to; 
    \item The columns of $L$ are orthogonal, hence $L^\top L$ is diagonal with the diagonal entries being $\Pa{L^\top L}_{i,i} = \sum_{t \in \{t \in \dbrack{\Nn}, i \in G_t\}} w_t^2,~ i \in \dbrack{n}$;
    \item For $k \in \dbrack{p}$, denote $t = \phi(k)$ according to \cref{eq:phi_k}. By \BLUE{(a)}, denote the column index of the unique non-zero entry in the $k$-th row as $j_k$, then $L_{k,j_k} \neq 0 $ is the unique non-zero entry in $L_{J_t,j_k}$ by \BLUE{(b)}, and $j_k \in G_t$.
\end{enumerate}
\end{fact}

\subsection{Properties of the projection operators}
\label{sec:prop:proj}
For an overlapping group sparse $x \in \RR^n$, we have the following results of the projection operators of the subspaces defined in \cref{sec:def_extended_supp}.

\begin{lemma}
\label{lemma:proj}
Define $\Ii_{x}, \Ee_x,\Ee_z,\Ee_L,\Tt_x,\Tt_z,\Tt_L$ as in \cref{eq:support_Ix}, \cref{eq:support_Ex}, \cref{eq:support_Ez} and \cref{eq:support_EL}, there holds
    \begin{enumerate}[label={\rm \alph*).}]
    \item $\Tt_L \subseteq \Tt_{z}$.
        \item $P_{\Tt_{L}^\bot} L P_{\Tt_x} = 0$, $ P_{\Tt_{L}} LP_{\Tt_x^\bot} = 0$, $P_{\Tt_{L}^\bot}L = L P_{\Tt_{x}^\bot}$.

        \item $\forall x'\in\RR^n$, $P_{\Tt_{z}} L P_{\Tt_x}x' = LP_{\Tt_x} x'$ and $P_{\Tt_{z}^\bot} L P_{\Tt_{x}} = 0$; 
        \item $\forall u \in \RR^p,P_{\Tt_x} L^\top P_{\Tt_L} u=P_{\Tt_x} L^\top u$ and $ P_{\Tt_x} L^\top P_{\Tt_{z}} u = P_{\Tt_x} L^\top u$;

        \item $L^\top P_{\Tt_{z}} L$ is diagonal with $\pa{L^\top P_{\Tt_{z}} L}_{i,i} = \ssum_{t\in \{ t\in \Ii_{x} \mid i \in G_t\}} w_t^2$. 
    \end{enumerate}
\end{lemma}
\begin{proof}
To verify \BLUE{(a)}, it suffices to prove $\Ee_L \subseteq \Ee_z$. For any $k \in \dbrack{p}$, denote $j_k$ the nonzero column index of $L$ in the $k$-th row, we have
\vspace*{-1ex}
\[
    \Pa{LP_{\Tt_x} \bm{1}_n}_k = \msum_{j \in \Ee_x} L_{i,j} = \left\{
    \begin{aligned}
        &w_{\phi(k)}, & j_k \in \Ee_x, \\
        &0, & j_k \notin \Ee_x.
    \end{aligned}
    \right.
\]
Hence for any $k \in \Ee_L = \supp(L{P_{\Tt_x}}\bm{1}_n)$, it follows that $j_k \in \Ee_x$. Recall \cref{fact:L} \BLUE{(d)} that $j_k \in G_{\phi(k)}$, therefore ${\phi(k)} \in \Ii_x$, which implies that $k \in \Ee_z$.

For \BLUE{(b)}, given any $x'\in \RR^n$, $\supp(L P_{\Tt_x} x') \subseteq \Ee_L$, hence $L P_{\Tt_x} x' \in \Tt_{L}$, which implies that $P_{\Tt_{L}^\bot} L P_{\Tt_x} = 0$. 
Note that $P_{\Tt_L} L P_{\Tt_x^\bot} \bm{1}_n \in \Tt_L$, by definition of $\Tt_L$
\vspace*{-1.5ex}
\begin{equation*}
    \supp(P_{\Tt_L} L P_{\Tt_x^\bot} x') \subseteq 
    \supp(P_{\Tt_L} L P_{\Tt_x^\bot} \bm{1}_n)
    \subseteq
    \Ee_L = \supp(L P_{\Tt_x} \bm{1}_n).
\end{equation*}
On the other hand, since $P_{\Tt_L}$ is binary diagonal matrix, we have
\vspace*{-1ex}
\begin{equation*}
    \supp(P_{\Tt_L} L P_{\Tt_x^\bot} x') \subseteq 
    \supp(P_{\Tt_L} L P_{\Tt_x^\bot} \bm{1}_n)
    \subseteq 
    \supp(L P_{\Tt_x^\bot} \bm{1}_n) .
\end{equation*}
Note that $\supp(L P_{\Tt_x^\bot} \bm{1}_n)\cap \supp(L P_{\Tt_x} \bm{1}_n)=\varnothing$, hence $\supp(P_{\Tt_L} L P_{\Tt_x^\bot} x') = \varnothing$, meaning that $P_{\Tt_L} L P_{\Tt_x^\bot} x'=0$, hence $P_{\Tt_L} L P_{\Tt_x^\bot}=0$. The last assertion of \BLUE{(b)} can be deduced by the previous two equations. 

To prove \BLUE{(c)}, since $L P_{\Tt_x}x' \in \Tt_{L} \subseteq \Tt_{z}$, it holds that $P_{\Tt_{z}} L P_{\Tt_x}x' = LP_{\Tt_x} x'$ and $P_{\Tt_{z}^\bot} L P_{T_{x}} = 0$.
%
Moreover, note that $ P_{\Tt_x} L^\top P_{\Tt_{L}^\bot} = 0$ (by \BLUE{(b)}) and $P_{\Tt_{x}}L^\top P_{\Tt_z^\bot}=0$ (by \BLUE{(c)}),
which verifies \BLUE{(d)}. 
The last claim can be obtained by direct calculation. For any $i\in\dbrack{n}$, we have $\Pa{L^\top P_{\Tt_{z}} L}_{i,i} =\ssum_{k \in \Ee_z} L_{k,i}^2  = \ssum_{t \in \{ t \in \Ii_{x} \mid i \in G_t\}} w_t^2$. 
%
\end{proof}

\begin{lemma}\label{lem:lemma2.4}
If $u \in \partial \norm{Lx}_{1,2}$, then it holds that $P_{\Tt_{L}^\bot} P_{\Tt_z} u = 0$.
\end{lemma}
\begin{proof}
Since $u \in \partial \norm{Lx}_{1,2}$, given $t \in \Ii_x$, $u_{J_t} = x_{G_t} /\norm{x_{G_t}}$, we have $\supp(u)_{J_t} = \supp(L P_{\Tt_x} \bm{1}_n)_{J_t}$, hence $\supp(P_{\Tt_z}u) \subseteq \supp(L P_{\Tt_x}\bm{1}_n) = \Ee_L$.
\end{proof}

\subsection{Properties of the effective lifting operator}
\label{appendix:widehat}

We isolate the full column-rank property of $\widehat{L}$ as a lemma. 

\begin{lemma}\label{lem:hatL_full_rank}
$\widehat{L}$ has full column rank, and $\widehat{L}^\top \widehat{L}$ is diagonal and invertible. Moreover, there holds $L^\top \widehat{L} = \widehat{L}^\top \widehat{L}$.
\end{lemma}
\begin{proof}
By definition of $\widehat{L}$ we have
\begin{equation}\label{eq:hatL_hatL}
    \begin{aligned}
        \widehat{L}^\top \widehat{L} &= L^\top L - P_{\Tt_{x}^\bot} L^\top P_{\Tt_{z}} L - L^\top P_{\Tt_{z}} L P_{\Tt_{x}^\bot} + P_{\Tt_{x}^\bot} L^\top P_{\Tt_{z}} L P_{\Tt_{x}^\bot} \\ 
        &= L^\top L - P_{\Tt_{x}^\bot} L^\top P_{\Tt_{z}} L .
    \end{aligned}
\end{equation}
Using \BLUE{(e)} of \cref{lemma:proj}, the diagonal entries of $\widehat{L}^\top \widehat{L}$ are
\begin{equation}
\label{eq:W_ii}
    \pa{ \widehat{L}^\top \widehat{L} }_{(i,i)}
=\left\{
\begin{aligned}
\ssum_{ \ba{ t\in\dbrack{\Nn} \mid i \in G_{t} } } w_t^2  ,&\quad  i \in  {\Ee_{x}} , \\
\ssum_{ \ba{ t\in \Ii_{x}^c \mid i \in G_{t}  }  } w_t^2  ,&\quad  i \in  \Ee_{x}^c  ,
\end{aligned}
\right.
\end{equation}
by the definition of $\Ee_x^c$, $\ba{t\in \Ii_{x}^c \mid i \in G_{t}}$ is non-empty when $i \in \Ee_x^c$, which implies $\widehat{L}^\top \widehat{L}$ is invertible. Note that 
\begin{equation*}
    L^\top \widehat{L} = L^\top (L - P_{\Tt_{z}} L P_{\Tt_{x}^\bot}) = L^\top L - L^\top P_{\Tt_{z}} L P_{\Tt_{x}^\bot},
\end{equation*}
combining with \cref{eq:hatL_hatL} concludes the proof.
\end{proof}

\begin{proof}[Proof of \cref{rem:widehat_eq}]
 Based on \cref{lem:lemma2.4} and \cref{lemma:proj} \BLUE{(b)}, for any $u \in \partial \norm{Lx}_{1,2}$, we have $(P_{\Tt_z} L P_{\Tt_x^\bot})^\top u=L^\top(P_{\Tt_{L}^\bot} P_{\Tt_z}) u=0$, therefore
    $L^\top u = (L - P_{\Tt_z} L P_{\Tt_x^\bot})^\top u$. Since $\widehat{L} = L - P_{\Tt_{z}}LP_{\Tt_{x}^\bot}$, we obtain
    \begin{equation*}
        L^\top u = \widehat{L}^\top u  ,
    \end{equation*}
    which concludes the proof.
\end{proof}


\section{Proof of certification property of OGN certificate}
\label{appendix:ogn_cert}

\begin{proof}[Proof of \cref{prop:ogn_certificate}]
We first show that there holds $L^\top u^\dagger = \betasol$. 
By definition of $u^\dagger$, we have
\begin{align*}
    L^\top u^\dagger &= L^\top (P_{\Tt_{\zsol}} u^\dagger + P_{\Tt_{\zsol}^\bot} u^\dagger) \\
    &= L^\top (P_{\Tt_{\zsol}} \usol + P_{\Tt_{\zsol}^\bot} \widehat{u}_{\min}) \\
    &= P_{\Tt_{\xsol}} L^\top P_{\Tt_{\zsol}} \usol + L^\top P_{\Tt_{\zsol}^\bot} \widehat{u}_{\min} \doubleslash{$\pa{P_{\Tt_{\zsol}}LP_{\Tt_{\xsol}^\bot}}^\top \usol = 0$, \cref{rem:widehat_eq}}\\
    &= P_{\Tt_{\xsol}} L^\top \usol + P_{\Tt_{\xsol}^\bot} L^\top P_{\Tt_{\zsol}^\bot} \widehat{u}_{\min} \doubleslash{$P_{\Tt_{\xsol}} L^\top P_{\Tt_{\zsol}^\bot} = 0$, \cref{lemma:proj}~(c)}\\
    &= P_{\Tt_{\xsol}} \betasol  + P_{\Tt_{\xsol}^\bot} L^\top P_{\Tt_{\zsol}^\bot} \widehat{u}_{\min} ,
\end{align*}
which implies that we need to show $ P_{\Tt_{\xsol}^\bot} L^\top P_{\Tt_{\zsol}^\bot} \widehat{u}_{\min} = P_{\Tt_{\xsol}^\bot} \betasol$. Note that 
\begin{equation}
\label{eq:Ltopumin}
    P_{\Tt_{\xsol}^\bot}L^\top P_{\Tt_{\zsol}^\bot} \widehat{u}_{\min} = P_{\Tt_{\xsol}^\bot}L^\top  \widehat{u}_{\min} - P_{\Tt_{\xsol}^\bot}L^\top P_{\Tt_{\zsol}} \widehat{u}_{\min}.
\end{equation}
For the right hand side of the equality,
\begin{itemize}
\item For the first term, combining $L^\top \widehat{L} = \widehat{L}^\top \widehat{L}$ from \cref{lem:hatL_full_rank} and the definition of $\widehat{u}_{\min}$ \eqref{eq:umin}, we get
\[
P_{\Tt_{\xsol}^\bot}L^\top  \widehat{u}_{\min} = P_{\Tt_{\xsol}^\bot}L^\top \widehat{L} (\widehat{L}^\top \widehat{L})^{-1} \betasol = P_{\Tt_{\xsol}^\bot}\betasol.
\]

\item 
For the second term, note that
\begin{align*}
P_{\Tt_{\xsol}^\bot} L^\top P_{\Tt_{\zsol} } \widehat{u}_{\min}
&= 
P_{\Tt_{\xsol}^\bot} L^\top P_{\Tt_{\zsol} } \widehat{L} \pa{ \widehat{L}^\top \widehat{L} }^{-1} \betasol \\
&=
P_{\Tt_{\xsol}^\bot} L^\top P_{\Tt_{\zsol} } \pa{L - P_{ \Tt_{\zsol}}L P_{\Tt_{\xsol}^\bot}} \pa{ \widehat{L}^\top \widehat{L} }^{-1} \betasol \doubleslash{definition of $\widehat{L}$} \\
&=P_{\Tt_{\xsol}^\bot} L^\top P_{\Tt_{\zsol} } \pa{L - L P_{\Tt_{\xsol}^\bot}} \pa{ \widehat{L}^\top \widehat{L} }^{-1} \betasol \\
&=
P_{\Tt_{\xsol}^\bot} L^\top P_{\Tt_{\zsol} } \pa{L P_{\Tt_{\xsol}}} \pa{ \widehat{L}^\top \widehat{L} }^{-1} \betasol \\
&= 0. \doubleslash{$L^\top P_{\Tt_{\zsol}} L$ is diagonal}
\end{align*}

\end{itemize}
Summarizing the above results concludes the proof of $L^\top u^\dagger = \betasol$.

For the certification property of $u^\dagger$, suppose $\norm{u^\dagger_{J_t}} < 1$. If $\xsol_{G_t} \neq 0$, according to \cref{def:ogn-cert}, $\norm{u^\dagger_{J_t}} = 1$, which is a contradiction.
\end{proof}

\begin{proof}[Proof of \cref{prop:link_two}]
Note for $t\in\Ii_{\xsol}^c$, $G_t \subseteq \Ee_{\xsol}^c, \xsol_{G_t} = 0$ and $u^\dagger_{J_t} = (\widehat{u}_{\min})_{J_t}$. 
%
Owing to \cref{lem:hatL_full_rank} that $\widehat{L}^\top \widehat{L}$ is diagonal, and for any $i\in G_t$,
\[
\pa{ \widehat{L}^\top \widehat{L} }_{(i,i)}
=
\msum_{ \ba{ j \in \Ii_{\xsol}^c \mid i \in G_{j}  }  } w_j^2 
\geq w_t^2 
\quad\Longrightarrow\quad
\pa{ \widehat{L}^\top \widehat{L} }_{(i,i)}^{-1} \leq \sfrac{1}{w_t^2}  . 
\]
Let $\widehat{L}_{G_t}$ be the submatrix of $\widehat{L}$ by selecting the columns of $\widehat{L}$ via $G_t$, then
\begin{align*}
    u^\dagger_{J_t}
= \Pa{\widehat{L} \pa{ \widehat{L}^\top \widehat{L} }^{-1} \betasol}_{J_t}
&= \Pa{\pa{L - P_{\Tt_{\zsol}}LP_{\Tt_{\xsol}^\bot}} \pa{ \widehat{L}^\top \widehat{L} }^{-1} \betasol}_{J_t} \\
& = \Pa{L \pa{ \widehat{L}^\top \widehat{L} }^{-1} \betasol}_{J_t} \doubleslash{$\forall x \in \RR^n, (P_{\Tt_{\zsol}} Lx)_{J_t} = 0$ for $t \in \Ii_{\xsol}^c$} \\
&= w_t \pa{ \widehat{L}^\top_{G_t} \widehat{L}_{G_t} }^{-1} \betasol_{G_t} .\doubleslash{Definition of $L$, \cref{eq:zJi_eq}}
\end{align*}
Consequently
\[
\begin{aligned}
\norm{u^\dagger_{J_t}}
= w_t \norm{ \pa{ \widehat{L}^\top_{G_t} \widehat{L}_{G_t} }^{-1} \betasol_{G_t} } 
\leq w_t \sfrac{1}{w_t^2} \norm{ \betasol_{G_t} } 
= \sfrac{1}{w_t} \norm{ \betasol_{G_t} }  ,
\end{aligned} 
\]
which concludes the proof. 
\end{proof}

\section{Tightness of the LASSO certificate}
\label{appendix:relaxation}

From the proof of \cref{prop:lasso_certificate}, relaxations in inequalities are required, see e.g. \cref{eq:beta_Gt}. This implies that the LASSO certificates are not as tight as those for the nonoverlapping case. Therefore, in the following we discuss the tightness of the LASSO certificate $\betasol$, which relies on the two discarded terms in \cref{eq:beta_Gt}
\begin{equation}\label{eq:two_terms}
  T_1\eqdef\msum_{i\in G_t, \xsol_i = 0} \pa{ \betasol_{i} }^2  \qandq T_2\eqdef\msum_{i\in G_t,\xsol_i\neq 0} \msum_{ k\in \mathscr{K}_i\setminus \ba{k_i}, j = \phi(k)  } \pa{  w_j \usol_k }^2 .
\end{equation}
A simple example is provided in \cref{fig:two_terms} to illustrate these two terms. 

\begin{figure}[htb]
    \centering
    \includegraphics[height=41mm, trim={21mm 2.35mm 19mm 2mm},clip]{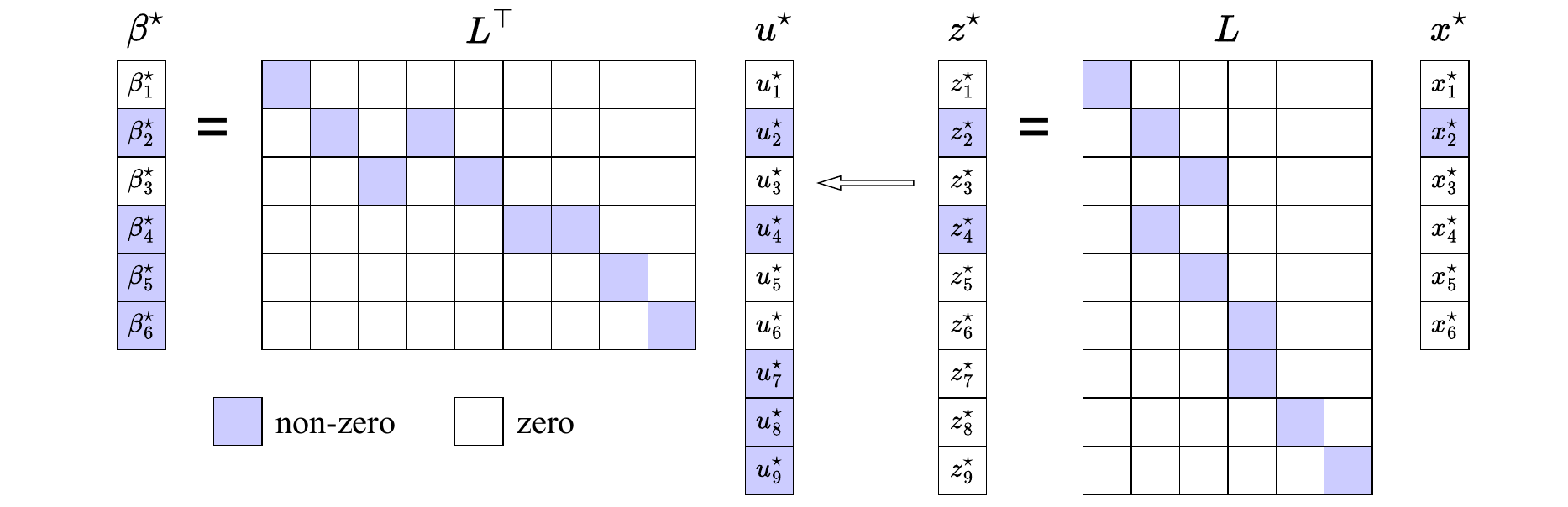} 
    \vspace*{-1ex}
    \caption{
    Illustration for the computation of the two discarded terms: i) $\xsol$ has only one nonzero entry; ii) The covering $\Gg = \Ba{ \ba{1,2,3}, \ba{2,3,4}, \ba{4,5,6} }$; iii) The index of the nonzero groups of $x$ is $\Ii_x = \Ba{1,2}$; iv) The extended support of $x$ and $z$ are $\Ee_x = \Ba{1,2,3}$ and $\Ee_z = \Ba{1,2,3,4,5,6}$. 
    } \vspace*{-2ex}
    \label{fig:two_terms}
\end{figure}

{\bf Term $T_1$:} For the purpose of simplicity, suppose ${G_t}$ only intersects with one nonzero group, say ${G_{t+1}}$, with the overlapping entries being nonzero. Let 
$
O_{t} \eqdef  G_t \mcap  G_{t+1} ,
$ 
for each $i\in O_{t}$, the set $\mathscr{K}_{i}$ in \cref{eq:Ki} contains only two elements that are
\[
 k_{i, t} \in J_{t}
 \qandq
 k_{i, t+1} \in J_{t+1}   .
\]
Back to the term $T_1$, since now $k_i = k_{i, t}$, we have the following simplification
\[
\begin{aligned}
\msum_{i\in G_t,\xsol_i\neq 0} \msum_{ k\in \mathscr{K}_i\setminus \ba{k_i}, j = \phi(k)  } \pa{  w_j \usol_k }^2
&= \msum_{i\in O_{t},\xsol_i\neq 0} \pa{  w_{t+1} \usol_{k_{i, t+1}} }^2 \\
&= w_{t+1}^2 \msum_{i\in O_{t}} \pa{   \usol_{k_{i, t+1}}  }^2  \\
&= w_{t+1}^2 { \norm{\xsol_{O_{t}}}^2 }/{ \norm{\xsol_{G_{t+1}}}^2 }  .
\end{aligned} 
\]
Here the ratio ${ \norm{\xsol_{O_{t}}}^2 }/{ \norm{\xsol_{G_{t+1}}}^2 } $ represents the energy of the overlapping part against the overlapping group, clearly the larger the ratio the less tight the bound in \cref{eq:beta_Gt}.

{\bf Term $T_2$:} Let $t\in\Ii_{\xsol}$, for each $i\in G_t$ with $\xsol_i=0$, there are two cases for $\betasol_i$
\begin{enumerate}[label={\rm \alph{*}).}]
    \item Entry $\xsol_i$ is either not in the overlapping region ($i=1$ in \cref{fig:two_terms}) or in the overlapping region with only nonzero groups ($i=3$ in \cref{fig:two_terms}), the corresponding entry in $\usol$ is zero, hence $\betasol_i = 0$. 

    \item Entry $\xsol_i$ also belongs to a zero group, then in general $\betasol_i\neq 0$. In \cref{fig:two_terms}, $\usol_{6}, \usol_{7}$ correspond to $\xsol_4$ with $\usol_6=0$, $\usol_{7}$ in general is not zero as it is a subgradient from zero group, hence $\betasol_4\neq 0$. 
\end{enumerate}
If we further combine the definitions of $\mathscr{K}_i$ in \cref{eq:Ki} and $\phi(\cdot)$ in \cref{eq:phi_k}, then for $T_2$
\[
\begin{aligned}
\msum_{i\in G_t, \xsol_i = 0} \pa{ \betasol_{i} }^2
&= \msum_{i\in G_t, \xsol_i = 0} \bPa{ \msum_{ k\in \mathscr{K}_i, j = \phi(k)  } w_{j} \usol_k  }^2 \\
&= \msum_{i\in G_t, \xsol_i = 0} \bPa{ \msum_{ k\in \mathscr{K}_i, j = \phi(k), \xsol_{G_j} = 0  } w_j \usol_k  }^2 ,
\end{aligned}
\]
which implies that {the value of $T_2$ is controlled by the zero groups that are overlapping with the current nonzero group.}

\begin{remark}
    Recall from \cref{eq:subdiff_l2}, that the choice of $\usol$ for the zero groups is such that $\betasol = L^\top \usol$, therefore in general we cannot further control term $T_2$. 
\end{remark}

\begin{small}
\bibliographystyle{siamplain}
\bibliography{reference}
\end{small}

\end{document}